\documentclass{amsart}
\usepackage{amsmath}
\usepackage{amsfonts}
\usepackage{graphicx, subfigure}
\usepackage[english]{babel}
\usepackage{amsmath,amssymb}
\usepackage{amsthm}
\usepackage[usenames,dvipsnames]{color}
\usepackage{accents}
\usepackage{xcolor}
\usepackage{mathtools}
\usepackage{multicol}
\usepackage{comment}
\usepackage{caption}
\usepackage[inline]{enumitem}
\usepackage{longtable}
\usepackage{overpic}

\usepackage[
    backend=biber,
    style=numeric,
    sorting=none,
    url=false
]{biblatex}
\usepackage{csquotes}

\allowdisplaybreaks

\usepackage{standalone}
\usepackage{tikz}
\usetikzlibrary{shadows}
\usetikzlibrary{arrows}
\usetikzlibrary{arrows.spaced}
\usetikzlibrary{arrows.meta}
\usetikzlibrary{shapes.geometric}
\usetikzlibrary{fit}
\usetikzlibrary{backgrounds}
\usetikzlibrary{calc}

\definecolor{blue}{RGB}{0,32,91}
\definecolor{cyan}{RGB}{24,176,226}
\definecolor{lightgrey}{RGB}{199,201,199}
\definecolor{grey}{RGB}{85,85,85}
\definecolor{red}{RGB}{215,51,103}
\definecolor{skyblue}{RGB}{0,127,185}
\definecolor{green}{RGB}{164,96,36}
\definecolor{orange}{RGB}{242,149,18}
\definecolor{purple}{RGB}{169,57,131}

\usepackage[pdftex,hidelinks]{hyperref}
\usepackage[capitalise,nameinlink]{cleveref}

\newcommand{\la}{\lambda}

\DeclareMathOperator{\Fix}{Fix}

\theoremstyle{plain}
\newtheorem{thr}{Theorem}[section]
\Crefname{thr}{Theorem}{Theorems}

\newtheorem{lem}[thr]{Lemma}

\newtheorem{cor}[thr]{Corollary}
\Crefname{cor}{Corollary}{Corollaries}

\theoremstyle{definition}

\crefname{defi}{definition}{definitions}
\newtheorem{ex}[thr]{Example}
\theoremstyle{remark}
\newtheorem{remk}{Remark}
\theoremstyle{remark}
\crefname{remk}{remark}{remarks}

\newcommand{\field}[1]{\mathbb{#1}}
\newcommand{\R}{\field{R}}
\newcommand{\N}{\field{N}}

\newcommand{\VV}{\mathcal{V}}
\newcommand{\HH}{\mathcal{E}}
\newcommand{\ord}[2]{#1^{(#2)}}
\newcommand{\HHo}[1]{\ord{\mathcal{E}}{#1}}
\newcommand{\he}[2]{(\{#1\}, \{#2\})}
\newcommand{\uhe}[1]{\{#1\}}
\newcommand{\go}[1]{\ord{G}{#1}}

\newcommand{\goe}[1]{\go{#1}_e}
\newcommand{\ao}[1]{\ord{a}{#1}}

\newcommand{\x}{\mathbf{x}}
\newcommand{\FF}{\mathbf{F}}

\newcommand{\sset}[1]{\left\{#1\right\}}
\newcommand{\sm}{\smallsetminus}

\newcommand{\Gc}{\mathcal{G}}
\newcommand{\Hc}{\mathcal{H}}
\newcommand{\Vc}{\mathcal{V}}
\newcommand{\Ec}{\mathcal{E}}
\newcommand{\Xbb}{\mathbb{X}}

\newcommand{\Rd}{\R^d}
\newcommand{\Pf}{\mathfrak{P}}
\newcommand{\abs}[1]{\left|#1\right|}

\numberwithin{equation}{section}

\newcommand{\tr}{\mathsf{T}}

\title[Heteroclinic Dynamics and Higher-Order Interactions]{Heteroclinic Dynamics in Network Dynamical Systems with Higher-Order Interactions}
\date{\today}
\author{Christian Bick}
\address[C.~Bick]{Department of Mathematics, Vrije Universiteit Amsterdam, De Boelelaan~1111, 1081 HV Amsterdam, The Netherlands}
\address{Department of Mathematics, University of Exeter, Exeter EX4 4QF, United Kingdom}
\address{Mathematical Institute, University of Oxford, Oxford OX2 6GG, United Kingdom}
\address{Institute for Advanced Study, Technical University of Munich, Lichtenbergstr 2, 85748 Garching, Germany}
\email{c.bick@vu.nl}

\author{Sören von der Gracht}
\address[S.~von~der~Gracht]{Institute of Mathematics, Paderborn University, Warburger Str.~100, 33098 Paderborn, Germany}
\email{soeren.von.der.gracht@uni-paderborn.de}
\date{\today}

\excludecomment{percent}

\begin{document}

\maketitle

\begin{abstract}
Heteroclinic structures organize global features of dynamical systems. 
We analyze whether heteroclinic structures can arise in network dynamics with higher-order interactions which describe the nonlinear interactions between three or more units.
We find that while commonly analyzed model equations such as network dynamics on undirected hypergraphs may be useful to describe local dynamics such as cluster synchronization, they give rise to obstructions that allow to design heteroclinic structures in phase space.
By contrast, directed hypergraphs break the homogeneity and lead to vector fields that support heteroclinic structures.
\end{abstract}

%%%%%%
\section{Introduction}
\label{Sec:Intro}
\noindent
Networks of interacting dynamical units have been extremely successful to describe emergent collective dynamics---such as synchronization---seen in many real-world systems~\cite{Pikovsky2003,Strogatz2004}: The state of node $k\in\sset{1, \dotsc, N}$ is given by $x_k\in\Rd$ and evolves according to the network interactions.
If the interactions take place between pairs of nodes, then a graph~$\Gc=(\Vc, \Ec)$ is the traditional combinatorial object that captures the interaction structure, where each unit corresponds to a vertex $v\in\Vc$ and the (additive) pairwise interactions take place along edges $e\in\Ec$. 
However, recent work has highlighted the importance of ``higher-order'' nonadditive interactions between three or more units~\cite{Battiston2020,Bick2021,Battiston2021}: 
In analogy to dynamical systems on graphs, nonadditive interactions have been associated with hyperedges $e\in\Ec$ in a hypergraph~$\Hc = (\Vc,\Ec)$.
While numerous definitions of network dynamical systems on hypergraphs (whether undirected, directed, weighted, etc.) have appeared in the literature, this approach allows to link the associated network structure (a hypergraph or simplicial complex) to dynamical features (such as synchronization behavior).
Indeed, local dynamical features such as the existence and stability of (cluster) synchronization can be phrased in terms of the higher-order interaction structure~\cite{Mulas2020,Salova2021a,Salova.2021b,Bick2020,Aguiar2020,Gallo2022,Nijholt.2022c,vonderGracht.2023}.

\newcommand{\Wu}{W^{\mathrm{u}}}
\newcommand{\Ws}{W^{\mathrm{s}}}

At the same time, to understand the dynamics of real-world systems it is essential to go beyond local dynamics and linear stability and understand global features of the network dynamics. 
Heteroclinic structures in phase space consist (in the simplest case) of equilibria~$\xi_q$ together with heteroclinic trajectories~$\gamma_{p,q}$ that lie in the intersection of the unstable manifold of~$\xi_p$ and the stable manifold of~$\xi_q$. They have received particular attention since they can organize periodic or chaotic dynamics of closed vector fields~\cite{Weinberger.2018}. 
Moreover, they have been associated with dynamics that show metastability, for example in neuroscience, where one observes discrete states (represented by~$\xi_q$) and transitions between them (along the heteroclinic connections)~\cite{Rabinovich2006,Ashwin2024}.
Indeed, given a heteroclinic structure (a set of equilibria with directed connections between pairs of equilibria) there are different ways to construct dynamical systems whose phase space has the desired heteroclinic structure~\cite{Ashwin2013,Aguiar.2011,Field.2015,Field2017}.
How these constructions are affected by considering vector fields that reflect specific network interactions has not been systematically investigated: Field remarked that heteroclinic structures can be realized in networks with pairwise interactions if the inputs are sufficiently heterogeneous (i.e., they need to be distinguishable)~\cite{Field.2015}.

Here, we analyze how higher-order interactions affect the emergence of heteroclinic structures in phase space.
Commonly considered network dynamics on hypergraphs naturally come with homogeneity assumptions on the coupling, which in turn may affect the emergence of heteroclinic dynamics: 
First, if hyperedges are sets then the order of the inputs should not matter and the corresponding coupling function needs to be symmetric in the arguments.
Moreover, higher-order interaction networks are typically considered on undirected hypergraphs, which impose even more constraints than directed hypergraphs as for each edge all units are affected in the same way by all other units that are contained in the edge.
Second, one typically considers hypergraphs with hyperedges of a single type; this constrains the coupling functions.

We systematically analyze how different types of interactions yield obstructions to constructing heteroclinic structures in phase space. 
Specifically, we consider heteroclinic structures in Lotka--Volterra type dynamical systems, which includes the classical Guckenheimer--Holmes example~\cite{Guckenheimer.1988}, and the construction of heteroclinic structures in network dynamical systems by Field~\cite{Field.2015} subject to constraints imposed by the network structure.
For example, network dynamics on undirected hypergraphs are typically too symmetric for Field's construction to apply (\Cref{thr:field-undirected}). 
By contrast, both---the Guckenheimer--Holmes cycle as well as Field's construction---can be realized in for network dynamics on directed hypergraphs. 
Interestingly though, additional restriction to specific types of interactions or coupling functions may again lead to the obstruction of one or the other construction. 
For example, $m$-uniform hypergraphs support the Guckenheimer--Holmes cycle but not Field's construction. On the other hand, coupling that does not explicitly depend on the state of the node makes the emergence of the Guckenheimer--Holmes cycle impossible but does not obstruct Field's construction---see \cref{sec:prelim} below for details on these types. 
Note that for all heteroclinic structures in consideration, we prove (or disprove) the existence of heteroclinic connections.
Here we focus on these fundamental existence properties rather than whether multiple connecting trajectories occur---i.e., the cleanliness of the heteroclinic structure~\cite{Field2017}---nor the stability properties the structures may have (see for example~\cite{Podvigina2011} and references therein).

This paper is organized as follows: In \Cref{sec:prelim}, we provide necessary preliminaries. In particular, \Cref{sec:HetCycles} contains a brief overview over heteroclinic structures as well as the Guckenheimer--Holmes cycle and Field's construction, while \Cref{Sec:Setup} introduces hypergraphs and the class of dynamics on those hypergraphs that we investigate in this article. In \Cref{Sec:Undirected}, we observe that undirected hypergraphs support neither of the two constructions. \Cref{sec:gh-directed,sec:oscar-directed} contain detailed expositions that the contrary is true for directed hypergraphs for both constructions respectively. In \cref{sec:LargerN}, we comment on the realizability of the two constructions in hypergraphs with more than three vertices. In \Cref{sec:HOI}, we classify several hypergraph dynamics from the literature according to the types investigated before to clarify which of those support heteroclinic dynamics according to the constructions. We conclude in \Cref{sec:Discussion} with a brief discussion and outlook.

\newcommand{\HC}[2]{[#1\to#2]}
\newcommand{\Cyc}{\mathsf{C}}

%%%%%%
\section{Heteroclinic Cycles and Higher-Order Interactions}
\label{sec:prelim}
\noindent
We now introduce heteroclinic dynamics on the one hand and network dynamical systems with higher-order interactions on the other. This will set the stage for the rest of the manuscript since showing what type of heteroclinic dynamics we may expect in network dynamical systems with higher-order interactions is the main topic of this paper.

%%%
\subsection{Robust Heteroclinic Cycles}
\label{sec:HetCycles}

Heteroclinic trajectories arise when the stable and unstable manifolds of distinct equilibria intersect. 
For the dynamical system $\dot x= f(x)$ on~$\R^n$ let~$\alpha(x)$, $\omega(x)$ be the usual~$\alpha$ and $\omega$~limit sets for the flow generated by~$f$ as $t\to\pm\infty$~\cite{Katok1995}. 
For a hyperbolic equilibrium $\xi\in\R^n$ we define
\begin{align*}
\Ws(\xi) &:= \{x\in\R^n:\omega(x)=\xi\}, & \Wu(\xi) &:= \{x\in\R^n:\alpha(x)=\xi\}
\end{align*}
to be its stable and unstable manifold, respectively.
A \emph{heteroclinic cycle~$\Cyc$} now consists of a finite number of hyperbolic equilibria~$\xi_q\in\R^n$, $q=1,\dotsc,Q$, together with heteroclinic trajectories
\[\HC{\xi_q}{\xi_{q+1}} \subset \Wu(\xi_q)\cap \Ws(\xi_{q+1})\neq\emptyset,\]
where indices are taken modulo~$Q$.
See~\cite{Weinberger.2018} for a recent overview of heteroclinic dynamics including heteroclinic cycles between more general invariant sets and larger heteroclinic structures that contain more than one distinct cycle.\footnote{Such heteroclinic structures are typically called \emph{heteroclinic networks}. To avoid confusion with network dynamical systems that determine the class of vector fields we consider, we avoid the term heteroclinic network and talk about heteroclinic structures instead.}

Heteroclinic cycles do not persist under generic perturbations of the vector field. Hence, one often considers heteroclinic cycles such that the heteroclinic trajectories $\HC{\xi_p}{\xi_q}$ are contained in dynamically invariant subspaces; these heteroclinic cycles are then \emph{robust} with respect to perturbations that preserve the invariant subspaces.

An important class of such dynamical systems is given by vector fields that are equivariant with respect to a symmetry group. Let $\dot x = f(x)$ with vector field $f:\R^n\to\R^n$ determine the dynamics and suppose that a group~$\Gamma$ acts on~$\R^n$. 
If the vector field commutes with the action of~$\Gamma$, that is, $\gamma f = f\gamma$, then~$f$ is \emph{$\Gamma$-equivariant} and~$\Gamma$ are symmetries of the dynamical system; see~\cite{Golubitsky2002} for an introduction to equivariant dynamical systems. 
As a consequence, for any subgroup $\Gamma'\subset\Gamma$ the fixed point subspace $\Fix(\Gamma')=\{x\in\R^n: \gamma x=x \text{ for all }\gamma\in\Gamma'\}$ is dynamically invariant. If all heteroclinic connections of a cycle are now contained in fixed point subspaces of the symmetry action, then the heteroclinic cycle is robust with respect to perturbations to the vector field~$f$ that preserve $\Gamma$-equivariance.

%%%
\subsubsection{Guckenheimer--Holmes Cycle}
\label{subsubsec:gh}

Guckenheimer and Holmes~\cite{Guckenheimer.1988} famously considered the heteroclinic cycle that arises in the 
system
\begin{equation}
    \label{eq:gh_cubic}
    \begin{split}
        \dot{x}_1 &= x_1 + ax_1^3 + bx_1x_2^2 + cx_1x_3^2 \\
        \dot{x}_2 &= x_2 + ax_2^3 + bx_2x_3^2 + cx_1^2x_2 \\
        \dot{x}_3 &= x_3 + ax_3^3 + bx_1^2x_3 + cx_2^2x_3
    \end{split}
\end{equation}
where $x=(x_1, x_2, x_3)\in\R^3$.
Note that these equations are equivariant with respect to the group $\Gamma = \langle\tau_1,\tau_2,\tau_3,\rho\rangle$, where the transpositions~$\tau_j$ act by sending $x_j\mapsto-x_j$ (keeping the other coordinates fixed) and the cyclic permutation $\rho:(x_1,x_2,x_3)\mapsto(x_2,x_3,x_1)$ permutes the coordinates. These equations can be interpreted as a \emph{network dynamical system} with three nodes (we will make this interpretation more precise below), where the state of node~$k\in\sset{1,2,3}$ is determined by $x_k\in\R$ and the interaction with the other nodes is determined by the parameters~$b,c$. As in Lotka--Volterra-type systems~\cite{Afraimovich2004}, the ``extinction subspaces'' $P_{12}, P_{23}, P_{31}$ where one of the coordinates is zero---e.g., $P_{12}=\Fix(\langle\tau_3\rangle)=\{x_3=0\}$---are dynamically invariant as fixed point subspaces of the transpositions.

The \emph{Guckenheimer--Holmes cycle} now connects the hyperbolic saddle equilibria
\begin{align*}
    \xi_1 &= (1/\sqrt{-a},0,0)^\tr,\\
    \xi_2 &= (0,1/\sqrt{-a},0)^\tr,\\
    \xi_3 &= (0,0,1/\sqrt{-a})^\tr
\end{align*}
that lie on the three coordinate axes; see \Cref{fig:HetCycles}(a) for a sketch of the heteroclinic cycle.
The cycle exists under the conditions $a+b+c=-1, -\frac{1}{3}<a<0, c<a<b<0$ and is robust as the heteroclinic connections are contained in the subspaces~$P_{kj}$.
Moreover, it attracts all trajectories that are not on the coordinate planes or the diagonals $\{x_1=\pm x_2=\pm x_3 \}$.

\begin{figure}
\begin{overpic}[width=0.7\linewidth]{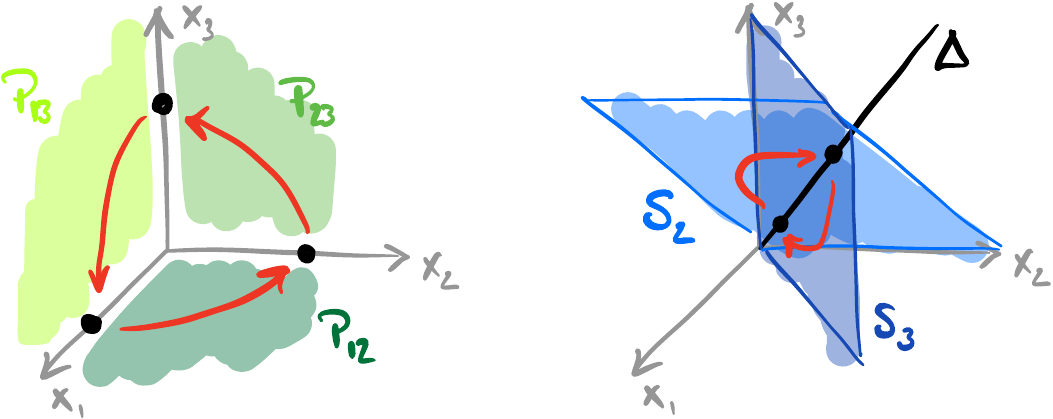}
        \put(-1,37){\textbf{(a)}}
        \put(53,37){\textbf{(b)}}
    \end{overpic}
\caption{\label{fig:HetCycles} Sketch of the Guckenheimer--Holmes (left) and Field cycle (right).}
\end{figure}

\newcommand{\Gr}{\mathsf{G}}
\newcommand{\Er}{\mathsf{E}}
\newcommand{\Vr}{\mathsf{V}}
%%%
\subsubsection{Constructing Robust Heteroclinic Cycles}
\label{subsubsec:fields-cycle}
While the Guckenheimer--Holmes cycle is an example of a specific heteroclinic cycle for a particular system, it is---more generally---possible to \emph{realize} general classes of graphs as heteroclinic structures in the phase space of a dynamical system~\cite{Ashwin2013,Field.2015}.
Here we consider the approach by Field~\cite{Field.2015} where a graph is realized as a heteroclinic structure in a network dynamical system (more specifically, a coupled cell system~\cite{Stewart2003}) in the following way: 
If~$\Gr=(\Vr,\Er)$ is the given graph to be realized as a heteroclinic structure with vertices~$\Vr$ and $N-1$ edges~$\Er$, consider a network dynamical system with~$N$ nodes such that the state of node~$k$ is given by $x_k\in\R$. 
Let $\Delta := \sset{(x_1, \dotsc, x_N)\in\R^N: x_1=x_2=\dotsb=x_N}$ denote the diagonal where all nodes are synchronized and let $S_j = \sset{(x_1, \dotsc, x_N)\in\R^N: x_l=x_k \text{ for }k,l\neq j}$ denote the set where all nodes except node~$j$ are synchronized.
For the specific class of systems in which the heteroclinic structure is realized, these subspaces are dynamically invariant. 
Then for each vertex $v_p\in\Vr$ there is a synchronized equilibrium $\xi_p\in\Delta$ and for each edge $(v_p,v_q)\in\Er$ there is a~$j(p,q)$ and a heteroclinic connection $\HC{\xi_p}{\xi_q}\subset S_{j(p,q)}$ along which the $j(p,q)$th node is not synchronized to the others.

For concreteness, we will focus on a minimal example of this general construction: 
We consider a heteroclinic cycle that consists of two equilibria~$\xi_1$ and~$\xi_2$ with reciprocal heteroclinic connections $C_1:=\HC{\xi_1}{\xi_2}$ and $C_2:=\HC{\xi_2}{\xi_1}$ in a network dynamical system that consists of $N=3$ nodes~\cite{Aguiar.2011}. In the following we will refer to this as the \emph{Field cycle}; see \Cref{fig:HetCycles}(b) for a sketch of the heteroclinic cycle.
Since the heteroclinic connection~$C_1$ lies in the invariant subspace $S_2 = \sset{x_1=x_3}$ and~$C_2$ in $S_3=\sset{x_2=x_3}$, the heteroclinic cycle is robust with respect to perturbations that preserve these invariant subspaces (e.g., perturbations with symmetry).

%%%%%%%
\subsection{Hypergraphs and Network Dynamics}
\label{Sec:Setup}

Hypergraphs as combinatorial objects are convenient to capture the interaction structure of network dynamical systems. A \emph{directed hypergraph~$\mathcal{H}=(\Vc,\Ec)$} on~$N$ vertices which consists of a set of vertices $\Vc = \sset{1, \dotsc, N}$ and each directed hyperedge $e\in\Ec$ (in the following simply \emph{edge}) can be written as $e=(T,H)$ with \emph{tail} $T=T(e)\in \Pf(\Vc)$ and \emph{head} $H=H(e)\in \Pf(\Vc)$. 
(We discuss the special case of an undirected hypergraph below.)
We call $m=\abs{T(e)}+1=:|e|$ the \emph{order} of a hyperedge~$e$ and write~$\Ec^{(m)}$ for all hyperedges of order~$m$; obviously $\Ec=\bigcup_{m\geq2}\Ec^{(m)}$.
A hypergraph~$\Hc$ is \emph{$m$-uniform} if $\Ec = \Ec^{(m)}$.

%%%
\subsubsection{Network Dynamics on Hypergraphs}

We now consider network dynamical systems consisting of~$N$ nodes that are compatible with the network structure determined by a hypergraph~$\Hc$.
Suppose that the state of node~$k$ is given by $x_k\in \Xbb = \Rd$.
If uncoupled, each node will evolve according to their (identical) intrinsic dynamics, determined by $F:\Xbb\to\Xbb$.
For an edge~$e$ with $k\in H(e)$ the state~$x_k$ will be influenced by the states of the nodes in~$T(e)$ through an interaction function~$G_e: \Xbb^{1+\abs{T(e)}}\to\Xbb$.
For~$e=(T,H)$ and $k\in H(e)$ we write~$G_e(x_k; x_T)$.
{The function~$G_e$ captures dependencies so we assume that if $G_e\neq 0$ then it depends nontrivially on~$x_T$; we identify~$G_e=0$ with $e\not\in\Ec$.}
Since the hyperedges are sets by definition---which do not depend on the ordering of their elements---it is natural to assume some invariance properties of the interaction functions %~$G_e$. 
 and assume that~$G_e$ is invariant under permutations of the elements of~$x_T$. 
A coupling function is \emph{nodespecific} if the dependency on the first coordinate is nontrivial and \emph{nodeunspecific} otherwise, i.e., if $G_e(x_k; x_T)=G_e(x_T)$.
Note that a nodeunspecific coupling function does not imply that the coupling is not state dependent: 
The receiving node~$k$ can still be contained in the tail of the hyperedge.

Taken together, the state of node~$k$ for the network dynamical system on a hypergraph~$\Hc$ evolves according to
\begin{equation}
    \label{eq:hypernetdyn}
    \dot x_k = F(x_k) + 
    \sum_{e\in\Ec: k\in H(e)} G_e(x_k;x_{T}),
\end{equation}
where~$F:\Xbb\to\Xbb$ describes the intrinsic dynamics of each node and the coupling function $G_e:\Xbb^{m+1}\to\Xbb$ describes how nodes interact along the hyperedge~$e$. 
A common assumption is to assume that there is only a single interaction function for each hyperedge order~$m$, that is $G_e = \go{m}$ for all $e\in \Ec^{(m)}$---that is, the interaction is \emph{homogeneous in each order~$m$}. 
We will typically make this assumption.

Since it is much more common to consider dynamics on undirected hypergraphs, we briefly discuss this special case. 
A directed hypergraph~$\Hc$ is \emph{undirected} if all edges are of the form $e=(A,A)$ for $A\in\Pf(\VV)$. For notational simplicity we identify an undirected hyperedge $e=(A,A)$ with the set~$A$.
This means that for a network dynamical system~\eqref{eq:hypernetdyn}, each node in an undirected hyperedge gets the same input from all other nodes contained in the edge.\footnote{While this definition of undirectedness is not equivalent to the standard definition of an undirected graph if the maximal order of an edge is two, it relates to undirected graphs if one considers hyperedges of the form $(\{k,j\}, \{k,j\})$.}

%%%
\subsubsection{Interactions for One-Dimensional Node Dynamics}
\label{sec:HypDyn}

\newcommand{\nb}{\mathbf{n}}

Since the aim is to relate the emergence of heteroclinic cycles to higher-order interactions, we expand the network vector field~\Cref{eq:hypernetdyn}. Specifically, we focus on one-dimensional node dynamics\footnote{This is not a restriction as one can always see a network of~$N$ $d$-dimensional nodes as a network of~$Nd$ one-dimensional nodes.} $\Xbb = \R$. For order~$m$ interactions, that is, the state~$z\in\R$ of a node in the head is influenced by the states $y_1, \dotsc, y_{m-1}\in\R$ in the tail, we  expand the interaction functions~$G^{(m)}(z; y_1, \dotsc, y_{m-1})$ formally as
\begin{subequations}    
\label{eq:expansion}
\begin{align}
G^{(2)}(z; y_1) &= a^{(2)}_{1}(z)y_1 + a^{(2)}_{2}(z)y_1^2 + \dotsb\\
\begin{split}
G^{(3)}(z; y_1,y_2) &= a^{(3)}_{10}(z)y_1+ a^{(3)}_{01}(z)y_2\\
&\qquad
+ a^{(3)}_{20}(z)y_1^2+ a^{(3)}_{11}(z)y_1y_2+ a^{(3)}_{02}(z)y_2^2+\dotsb    
\end{split}\\
&\ \ \vdots\nonumber
\intertext{or written more compactly for order~$m+1$ with input vector $y\in\R^m$ and an $m$-dimensional multi-index $\nb=(n_1, \dotsc, n_m)\in\N^m$, $\abs{\nb}=n_1+\dotsb n_m$, $y^\nb=y_1^{n_1}\dotsb y_{\vphantom{1}m}^{n_m}$ as}\nonumber
G^{(m+1)}(z; y) &= \sum_{o=1}^{\infty}\sum_{\abs{\nb}=o}\ao{m+1}_\nb(z)y^\nb.
\end{align}
\end{subequations}
{First, note that the zeroth-order coefficient vanishes to ensure a nontrivial dependence on the states of nodes in the tail of the hyperedge.
Second, }the requirement that the functions are invariant under permutations of the inputs~$y$ imposes conditions on the coefficients $\ao{m+1}_\nb(z)$: For example, we have to have $\ao{3}_{010}(z)=\ao{3}_{001}(z)$. If the coupling functions are nodeunspecific, then the coefficients are constant with respect to the node state~$z$, that is, $\ao{m+1}_\nb(z)=\ao{m+1}_\nb$.

A particular choice of interaction function still leaves some ambiguity in terms of the network; see also~\cite{Aguiar2020} for a more detailed discussion. First, the interaction function $G^{(m)}(z; y_1, \dotsc, y_{m-1})$ may not depend on one (or more) of the~$y_l$ (as the relevant coefficients~$\ao{m+1}_\nb(z)$ vanish). This means that the interactions along $m$-dimensional edges are actually of lower order, say~$n$, and there is a possibility of two ``types'' of order~$n$ interactions, the ones determined by~$G^{(m)}$ and those by~$G^{(n)}$. Second, we do not impose that~$G^{(m)}$ is of minimal order~$m$. That means that $G^{(m)}(z; y_1, \dotsc, y_{m-1}) = y_1+\dotsb+y_{m-1}$---an interaction function that can be realized with a graph with pairwise edges---is a valid choice of interaction function for a hyperedge of order~$m$. In both cases, we say that the coupling (realized by the coupling function~$G^{(m)}$) is \emph{effectively} of a lower order~$n$.

%%%%%%%
\section{Obstruction to Heteroclinic Cycles for Network Dynamics on Undirected Hypergraphs}
\label{Sec:Undirected}
\noindent
In this section, we consider network dynamical systems with higher-order interactions on \emph{undirected} hypergraphs and see whether the heteroclinic cycles in Section~\ref{sec:HetCycles} can be realized in the resulting vector fields.
Recall that network dynamics on undirected hypergraphs have strong homogeneity properties properties: Each node in a hyperedge is affected by each other node in the same way.
As a result, we find that the undirected setup is quite restrictive as the resulting vector fields are constrained by the symmetries.

\newcommand{\Kc}{\mathcal{K}}

Consider a network dynamical system with one-dimensional node phase space (cf.~Section~\ref{Sec:Setup}) on an undirected hypergraph~$\Hc$ with three vertices $\Vc=\sset{1,2,3}$. As there are precisely~$7$ nontrivial undirected edges, the complete undirected hypergraph~$\Kc$ has edges
\begin{align*}
    \HHo{1} &= \{\uhe{1}, \uhe{2}, \uhe{3} \}, &
    \HHo{2} &= \{\uhe{1,2}, \uhe{1,3}, \uhe{2,3}\},\\
    \HHo{3} &= \{\uhe{1,2,3}\}, \text{ and}& 
    \HHo{\ell} &= \emptyset \text{ whenever } \ell >3.
\end{align*}
In the most general case, the coupling function may depend on the specific edge and the equations~\eqref{eq:hypernetdyn} for dynamics on~$\Kc$ read
\begin{equation}
    \label{eq:hypernetdyn-undirected}
    \begin{split}
        \dot x_1 &= F(x_1) + G_{\uhe{1}}(x_1;x_1) + G_{\uhe{1,2}}(x_1; x_1, x_2) + G_{\uhe{1,3}}(x_1;x_1,x_3) \\
            &\qquad + G_{\uhe{1,2,3}}(x_1;x_1,x_2,x_3) \\
        \dot x_2 &= F(x_2) + G_{\uhe{2}}(x_2;x_2) + G_{\uhe{1,2}}(x_2; x_1, x_2) + G_{\uhe{2,3}}(x_2;x_2,x_3) \\
            &\qquad  + G_{\uhe{1,2,3}}(x_2;x_1,x_2,x_3) \\
        \dot x_3 &= F(x_3) + G_{\uhe{3}}(x_3;x_3) + G_{\uhe{1,3}}(x_3; x_1, x_3) + G_{\uhe{2,3}}(x_3;x_2,x_3) \\
            &\qquad  + G_{\uhe{1,2,3}}(x_3;x_1,x_2,x_3).
    \end{split}
\end{equation}
We summarize the right hand side of~\eqref{eq:hypernetdyn-undirected} as~$\FF(x_1,x_2,x_3)$.

Of course, an undirected hypergraph on three vertices could also comprise only a subset of the edges presented above.
While this can be incorporated in the complete graph setup (e.g., by setting the coupling function of an edge to zero), we will comment explicitly below that this does not affect our main points.

%%%
\subsection{The Guckenheimer--Holmes Cycle} 
We fist consider the Guckenheimer--Holmes cycle in network dynamics on undirected hypergraphs.
\begin{thr}
    \label{thr:gh-undirected}
    The Guckenheimer--Holmes system \eqref{eq:gh_cubic} cannot be realized in a network dynamical system on an undirected hypergraph~$\Hc$ on three vertices~\eqref{eq:hypernetdyn-undirected}.
\end{thr}
\begin{proof}
    The Guckenheimer--Holmes cycle is a cycle between three equilibria that are related by a symmetry of the system (cf.~\Cref{subsubsec:gh}). Specifically, the vector field~\eqref{eq:gh_cubic} is equivariant with respect to the cyclic symmetries generated by the linear map $\rho(x_1,x_2,x_3)=(x_2,x_3,x_1)$.

    This places additional restrictions on the vector fields~\eqref{eq:hypernetdyn-undirected} for undirected hypergraph dynamics: For $\rho$, equivariance requires \begin{equation}
    \label{eq:GHundir}
        \FF(\rho(x_1,x_2,x_3))=\rho(\FF(x_1,x_2,x_3)).
    \end{equation}
    Writing each side of~\eqref{eq:GHundir} explicitly yields
    \begin{equation*}
    \FF(\rho(x_1,x_2,x_3)) = 
        \left(\begin{array}{l}
            F(x_2) + G_{\uhe{1}}(x_2;x_2) + G_{\uhe{1,2}}(x_2; x_2, x_3) \\ \qquad +\ G_{\uhe{1,3}}(x_2;x_2,x_1) + G_{\uhe{1,2,3}}(x_2;x_2,x_3,x_1) \\
            F(x_3) + G_{\uhe{2}}(x_3;x_3) + G_{\uhe{1,2}}(x_3; x_2, x_3) \\ \qquad +\ G_{\uhe{2,3}}(x_3;x_3,x_1) + G_{\uhe{1,2,3}}(x_3;x_2,x_3,x_1) \\
            F(x_1) + G_{\uhe{3}}(x_1;x_1) + G_{\uhe{1,3}}(x_1; x_2, x_1) \\ \qquad +\ G_{\uhe{2,3}}(x_1;x_3,x_1) + G_{\uhe{1,2,3}}(x_1;x_2,x_3,x_1)
        \end{array}\right),
    \end{equation*}
    while the right hand side is
    \begin{equation*}
    \rho(\FF(x_1,x_2,x_3)) = 
        \left(\begin{array}{l}
            F(x_2) + G_{\uhe{2}}(x_2;x_2) + G_{\uhe{1,2}}(x_2; x_1, x_2)\\ \qquad +\ G_{\uhe{2,3}}(x_2;x_2,x_3) + G_{\uhe{1,2,3}}(x_2;x_1,x_2,x_3) \\
            F(x_3) + G_{\uhe{3}}(x_3;x_3) + G_{\uhe{1,3}}(x_3; x_1, x_3)\\ \qquad +\ G_{\uhe{2,3}}(x_3;x_2,x_3) + G_{\uhe{1,2,3}}(x_3;x_1,x_2,x_3) \\
            F(x_1) + G_{\uhe{1}}(x_1;x_1) + G_{\uhe{1,2}}(x_1; x_1, x_2)\\ \qquad +\ G_{\uhe{1,3}}(x_1;x_1,x_3) + G_{\uhe{1,2,3}}(x_1;x_1,x_2,x_3)
        \end{array}\right).
    \end{equation*}
    So for~\eqref{eq:GHundir} to hold while using the fact that each interaction function~$G_e(x_k;x_{T(e)})$ is invariant under arbitrary permutations of the arguments in~$x_{T(e)}$, we obtain the restrictions
    \begin{align*}
        &G_{\uhe{1}}(z;y_1) = G_{\uhe{2}}(z;y_1) =          G_{\uhe{3}}(z;y_1), \\
        &G_{\uhe{1,2}}(z;y_1,y_2) = G_{\uhe{1,3}}           (z;y_1,y_2) = G_{\uhe{2,3}}(z;y_1,y_2).
    \end{align*}
    In particular, the coupling is homogeneous in every order $m=1,2,3$. Note that this observation does neither change when arbitrary hyperedges are not present nor when all coupling functions are nodeunspecific.

    As a result, each row of \eqref{eq:hypernetdyn-undirected} is invariant under permutation of the two input-variables---the $k$-th row is invariant under exchanging $x_i, x_j$ for $i,j\in\{1,2,3\}\setminus\{k\}$. This, however, is not true for the Guckenheimer--Holmes system \eqref{eq:gh_cubic} where the existence of the heteroclinic cycle depends crucially on $b\ne c$. Hence, this system cannot be realized by \eqref{eq:hypernetdyn-undirected}.
\end{proof}
\begin{remk}
    \label{remk:gh-apriori}
    The fact that the coupling is homogeneous in every order $m=1,2,3$ in the previous proof is a restriction that is imposed by the symmetry rather than a priori.
\end{remk}
The symmetry observations that are made for undirected hypergraphs can be summarized as follows: The Guckenheimer--Holmes system has a certain set of symmetries. On the other hand, restricting to undirected edges imposes additional symmetries on the system by the fact that certain terms are present in multiple equations. The combination of both results in the fact that there are in some sense `too many symmetries' present in the system for the heteroclinic cycle to emerge. On a more technical note, this is caused by the fact that the presence of all symmetries simultaneously yields that the equilibria in the cycle cannot have the necessary saddle stability properties.

%%%
\subsection{The Field Cycle}
\label{subsec:oscar-undirected}
Next, we investigate the Field cycle as an example of the more general construction in which heteroclinic cycles occur between fully synchronous equilibria in a system of three interacting nodes. 
Contrary to the Guckenheimer--Holmes cycle, the model in \cite{Aguiar.2011,Field2017,Weinberger.2018} does not specify a precise system of ordinary differential equations but rather proves that a given network structure allows for the realization of a heteroclinic cycle between two fully synchronous equilibria. 
While the key component for the Guckenheimer--Holmes cycle are the symmetries of the system, the main ingredient for this construction are specific synchrony subspaces that are dynamically invariant independent of the specific governing functions. 
In particular, the construction requires the dynamical invariance of the fully synchronous subspace $\Delta = \{x_1=x_2=x_3\}$ and of two of the partially synchronous subspaces $S_3 = \{x_1=x_2\}, S_2 = \{x_1=x_3\}$, and $S_1 = \{x_2=x_3\}$. 

\begin{thr}
    \label{thr:field-undirected}
    Dynamics on an undirected hypergraph on three vertices does not allow for the realization of the Field cycle.
\end{thr}
\begin{proof}
    We evaluate the restrictions on general dynamics on an undirected hypergraph~\eqref{eq:hypernetdyn-undirected} of three vertices imposed by the dynamical invariance of the (partial) synchrony subspaces. To that end, first consider the dynamics on $S_3 = \{x_1=x_2\}$ by substituting~$(x_1,x_2,x_3) = (z,z,y)$ into~\eqref{eq:hypernetdyn-undirected}. 
    The subspace is dynamically invariant if $\dot x_1=\dot x_2$, that is,
    \begin{align*}
        &F(z) + G_{\uhe{1}}(z;z) + G_{\uhe{1,2}}(z;z,z) + G_{\uhe{1,3}}(z;z,y) + G_{\uhe{1,2,3}}(z;z,z,y) \\
        &= F(z) + G_{\uhe{2}}(z;z) + G_{\uhe{1,2}}(z;z,z) + G_{\uhe{2,3}}(z;z,y) + G_{\uhe{1,2,3}}(z;z,z,y).
    \end{align*}
    Canceling out equal terms gives
    \[ G_{\uhe{1}}(z;z) + G_{\uhe{1,3}}(z;z,y) = G_{\uhe{2}}(z;z) + G_{\uhe{2,3}}(z;z,y). \]
    In particular, the coupling functions have to satisfy
    \[ G_{\uhe{1}} = G_{\uhe{2}} \quad \text{and} \quad G_{\uhe{1,3}} = G_{\uhe{2,3}}. \]
    Identical considerations for $S_2 = \{x_1=x_3\}$ and $S_1 = \{x_2=x_3\}$ yield 
    \begin{align*}
        &G_{\uhe{1}} = G_{\uhe{3}}, \quad G_{\uhe{1,2}} = G_{\uhe{2,3}} \quad \text{and} \\
        &G_{\uhe{2}} = G_{\uhe{3}}, \quad G_{\uhe{1,2}} = G_{\uhe{1,3}}
    \end{align*}
    respectively.

    Thus, invariance of any two of the subspaces~$S_k$, the coupling functions have to satisfy
    \[ G_{\uhe{1}} = G_{\uhe{2}} = G_{\uhe{3}} \quad \text{and} \quad G_{\uhe{1,2}} = G_{\uhe{1,3}} = G_{\uhe{2,3}}. \]
    As a result, the coupling is homogeneous in each order~$m$.
    Invariance of the fully synchronous subspace follows trivially from the intersection of the two-dimensional subspaces. 
    Again, note that none of these observations change when any hyperedges are not present (they could be represented by a coupling function~$0$) or when the coupling functions are all nodeunspecific.

    The goal of the construction is to construct heteroclinic connections in partially synchronous spaces between two hyperbolic equilibria in the fully synchronous subspace. In order for such a heteroclinic connection to be possible, equilibria in $\Delta = \{x_1=x_2=x_3\}$ need both a stable and an unstable direction outside of the fully synchronous subspace. 
    However, a quick calculation shows that the network structure in the equations is too restrictive for this to happen: The linearization of~\eqref{eq:hypernetdyn-undirected} at $\x=(z,z,z)^\tr \in \Delta$ is
    \begin{equation}
        \label{eq:oscar1-lin}
        D\mathbf{F}(\x) = \begin{pmatrix}
            \theta & \eta + \zeta & \eta + \zeta \\
            \eta + \zeta & \theta & \eta + \zeta \\
            \eta + \zeta & \eta + \zeta & \theta
        \end{pmatrix},
    \end{equation}
    where, with $\partial_j$ denoting differentiation with respect to the $j$th component,
    \begin{align*}
        \alpha &= F'(z), \\
        \beta &= \partial_1\go{1}(z;z), \\
        \gamma &= \partial_1 \go{2}(z;z,z), \\
        \delta &= \partial_1 \go{3}(z;z,z,z) \\
        \epsilon &= \partial_2 \go{1}(z;z), \\
        \eta &= \partial_2 \go{2}(z;z,z) = \partial_3       \go{2}(z;z,z), \\
        \zeta &= \partial_2 \go{3}(z;z,z,z) =               \partial_3 \go{3}(z;z,z,z) = \partial_4         \go{3}(z;z,z,z),\\
        \theta &= \alpha + \beta + 2 \gamma + \delta +      \epsilon + 2 \eta + \zeta. 
    \end{align*}
    This matrix has eigenvalues $\la_1=\theta+2(\eta + \zeta)$ and $\la_2=\la_3=\theta-(\eta+\zeta)$ with corresponding eigenvectors $(1,1,1)^\tr$ and $(1,-1,0)^\tr, (1,0,-1)^\tr$. In particular, in the generic situation where the two eigenvalues are not identical, any equilibrium in the fully synchronous space can either have a $2$-dimensional stable or a $2$-dimensional unstable manifold outside of~$\Delta$. Hence, reciprocal heteroclinic connections outside of~$\Delta$ are impossible.
\end{proof}
\begin{remk}
    Similar to \Cref{remk:gh-apriori}, the fact that the coupling functions need to be homogeneous in every order results from the restrictions on the vector field to realize the heteroclinic cycle rather than a priori. 
\end{remk}

%%%%%%%
\section{The Guckenheimer--Holmes Cycle in Directed Hypergraphs}
\label{sec:gh-directed}
\noindent
The situation for directed hypergraphs is more complex than the undirected case. On three vertices $\VV =\{1,2,3\}$ there are $49$ non-trivial directed hyperedges of order at most three---including the seven investigated in the previous section. 
Rather than investigating all possible configurations of hyperedges that exist in the hypergraph, we will  describe specific cases in which the heteroclinic constructions are possible due to the increased complexity. Specifically, we interpret the system~\eqref{eq:gh_cubic} as a (higher-order) network dynamical system with additive coupling as in~\eqref{eq:hypernetdyn}. This includes the construction of a suitable interaction structure as well as of the correct coupling functions. As opposed to dynamics on an undirected hypergraph, this construction is possible when directed hyperedges are present.

One of the major tools for the construction is the observation that system~\eqref{eq:gh_cubic} can be represented as
\begin{align*}
    \dot{x}_1 &= f(x_1; x_2, x_3) \\
    \dot{x}_2 &= f(x_2; x_3, x_1) \\
    \dot{x}_3 &= f(x_3; x_1, x_2),
\end{align*}
where
\begin{equation}
    \label{eq:GH-F}
    f(z, y_1, y_2) = z + az^3 + bzy_1^2 + czy_2^2.
\end{equation}
In particular, the system can be interpreted as a homogeneous system of three interacting dynamical vertices, as all three internal dynamics are governed by the same function. Hence, it suffices to construct suitable couplings for one vertex and extend the construction to the remaining ones. We can make one immediate observation.

\begin{thr}\label{thm:NodeunspecificNoGH}
    The Guckenheimer--Holmes system~\eqref{eq:gh_cubic} cannot be realized on a directed hypergraph~$\Hc$ on three vertices as a network dynamical system with nodeunspecific coupling.
\end{thr}
\begin{proof}
    The governing function~$f$ in the cubic Guckenheimer--Holmes system~\eqref{eq:GH-F} contains internal dynamics $z+az^3$ and two coupling terms~$bzy_1^2$ and~$czy_2^2$. Heuristically speaking, these couplings are nodespecific and thus cannot be realized by nodeunspecific coupling functions. More precisely, the coupling occurs in mixed terms of the state of the vertex~$z$ and the state one of its neighbors~$y_1$ or~$y_2$. Hence, a nodeunspecific coupling function would require the head vertex to be an element of the tail as well to be able to generate such a term, since it is of the form $\goe{m}(z;y_T)=\goe{m}(y_T)$. This can only be the case in hyperedges of order three or greater. If the term $bzy_1^2$ is contained in one of the nodeunspecific coupling functions, e.g., $\goe{3}(z,y_1)$ for a hyperedge $e\in\HHo{3}$, then so is $bz^2y_1$ due to the symmetry of the coupling function. This term then has to be removed by another coupling function, that is $-bz^2y_1$ has to be contained in another nodeunspecific coupling function e.g. $G^{(3)}_{e'}(z,y_1)$. But then, by the same symmetry argument, so is $-bzy_1^2$ and we also remove the wanted from the equation. This shows, that the function~$f$ cannot be realized by nodeunspecific coupling functions.
\end{proof}

%%%
\subsection{Realisation in a Directed Classical Network}
\label{subsubsec:gh-classical}
\begin{figure}
    \centering
    \resizebox{!}{.3\textwidth}{
        \begin{tikzpicture}[
%			background rectangle/.style={
%				fill=grey!10
%			},
%			show background rectangle,
%	->,
%	>=stealth',
%	shorten >=1pt,
%	shorten <=1pt,
	%		auto,
	%		node distance=5cm,
	square/.style = {
		regular polygon,
		regular polygon sides=4
	},
	main node/.style={
		%			node distance=2cm,
		line width=1.5pt, 
		circle,
		draw,
		font=\sffamily,
		inner sep=2pt,
		fill=white
	},
	second node/.style={
		%			node distance=2cm,
		line width=1.5pt, 
		square, 
		draw, 
		font=\sffamily,
		rounded corners, 
		inner sep=2pt
	},
	edge/.style={
		-stealth,
		shorten >=1pt,
		shorten <=1pt,
		line width=1.5pt
	},
	hyperedge/.style={
		Round Cap-{Triangle[length=3mm, width=2mm]},
		line width=5pt
	}
	]
	\node[rotate=50, regular polygon, regular polygon sides=3, minimum width=3cm] (tr) at (0,0) {};
	
	\node[main node] (1) at (tr.corner 1) {};
	\node[main node] (2) at (tr.corner 2) {};
	\node[main node] (3) at (tr.corner 3) {};
	
	%edges
	\path[line width=1.5pt]
	(2) edge [edge, bend left=10] (1)
	(3) edge [edge, bend left=10] (1)
	
	(1) edge [edge, bend left=10] (2)
	(3) edge [edge, bend left=10] (2)
	
	(1) edge [edge, bend left=10] (3)
	(2) edge [edge, bend left=10] (3)
	;
\end{tikzpicture}%
	}%
    \caption{A $3$-vertex graph that supports heteroclinic dynamics.}
    \label{fig:oscar}
\end{figure}
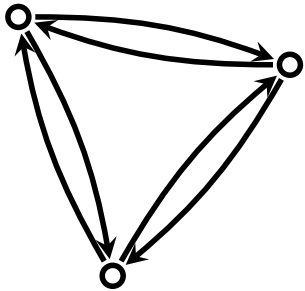

Observe that that governing function~$f$ contains effective pairwise interactions only.
Hence, one may want to realize the vector field as the network vector field of a classical network dynamical system with pairwise coupling.
For now we consider the most general case where the coupling function may depend on the edge; we will later discuss the strict constraints a homogeneity assumption would place on the dynamics.
The three subsystems in~\eqref{eq:gh_cubic} are all-to-all coupled. 
Hence, the network~$\Hc$ in consideration has to have edges
\begin{equation}
    \label{eq:gh-H1}
    \begin{split}
        \HHo{2} &= \{ \he{1}{1}, \he{1}{2}, \he{1}{3}, \\
                &\phantom{= \{} \he{2}{1}, \he{2}{2}, \he{2}{3}, \\
                &\phantom{= \{} \he{3}{1}, \he{3}{2}, \he{3}{3}
        \}, \text{ and}\\
        \HHo{m} &= \emptyset \text{ whenever } m >2,
    \end{split}
\end{equation}
as represented by \Cref{fig:oscar} (note that self-loops $\he{1}{1}, \he{2}{2}$, and $\he{3}{3}$ are not included in the figure).

\begin{thr}
    The (hyper)graph~$\Hc$ defined in \eqref{eq:gh-H1} allows for the realization of the Guckenheimer--Holmes system \eqref{eq:gh_cubic} as a network dynamical system if the coupling is not homogeneous.
\end{thr}
\begin{proof}
    Without loss of generality, we may realise terms that do not describe interactions via the internal dynamics exclusively by setting
    \[ F(z) = z + az^3. \]
    This would allow us to discard the self-loops from the network entirely. To generate the remaining terms $bzy_1^2 + czy_2^2$ as the sum of two pairwise coupling functions, we need two types thereof:
    \[ \go{2.1}(z;y_1) = bzy_1^2, \quad \go{2.2}(z;y_1) = czy_1^2. \]
    This yields
    \[ f(z,y_1,y_2) = F(z) + \go{2.1}(z;y_1) + \go{2.2}(z;y_2). \]
    To realize the explicit system \eqref{eq:gh_cubic} we obtain the desired vector field for
    \begin{align*}
        G_{\he{1}{2}}=G_{\he{2}{3}}=G_{\he{3}{1}} &= \go{2.1} ,\\
        G_{\he{1}{3}}=G_{\he{2}{1}}=G_{\he{3}{2}} &= \go{2.2}.
    \end{align*}
\end{proof}

The assumption of two different coupling functions can be regarded as two types of edges and the network would more precisely be represented as in \Cref{fig:gh_net}. If we assume homogeneity in the coupling, i.e., $G_e = \go{2}$ for all $e\in\HH$, this construction would force $b=c$. 
In this scenario, however, the equilibria are not saddles any longer and the heteroclinic cycle cannot exist.
Thus, the emergence of the Guckenheimer--Holmes cycle depends crucially on the fact that the governing function~$f$ may distinguish between the interaction of a given vertex with the other two. In a classical network with pairwise coupling, this requires two edge types and we say the network has \emph{asymmetric inputs}. 

\begin{figure}
    \centering
    \resizebox{!}{.3\textwidth}{
        \begin{tikzpicture}[
%			background rectangle/.style={
%				fill=grey!10
%			},
%			show background rectangle,
%	->,
%	>=stealth',
%	shorten >=1pt,
%	shorten <=1pt,
	%		auto,
	%		node distance=5cm,
	square/.style = {
		regular polygon,
		regular polygon sides=4
	},
	main node/.style={
		%			node distance=2cm,
		line width=1.5pt, 
		circle,
		draw,
		font=\sffamily,
		inner sep=2pt,
		fill=white
	},
	second node/.style={
		%			node distance=2cm,
		line width=1.5pt, 
		square, 
		draw, 
		font=\sffamily,
		rounded corners, 
		inner sep=2pt
	},
	edge/.style={
		-stealth,
		shorten >=1pt,
		shorten <=1pt,
		line width=1.5pt
	},
	hyperedge/.style={
		Round Cap-{Triangle[length=3mm, width=2mm]},
		line width=5pt
	}
	]
	\node[rotate=50, regular polygon, regular polygon sides=3, minimum width=3cm] (tr) at (0,0) {};
	
	\node[main node] (1) at (tr.corner 1) {};
	\node[main node] (2) at (tr.corner 2) {};
	\node[main node] (3) at (tr.corner 3) {};
	
	%edges
	\path[line width=1.5pt]
	(2) edge [edge, bend left=10] (1)
	(3) edge [edge, bend left=10, dashed] (1)
	
	(1) edge [edge, bend left=10, dashed] (2)
	(3) edge [edge, bend left=10] (2)
	
	(1) edge [edge, bend left=10] (3)
	(2) edge [edge, bend left=10, dashed] (3)
	;
\end{tikzpicture}%
	}%
    \caption{The $3$-vertex graph with two edge types that allows for the realization of the Guckenheimer--Holmes cycle.}
    \label{fig:gh_net}
\end{figure}
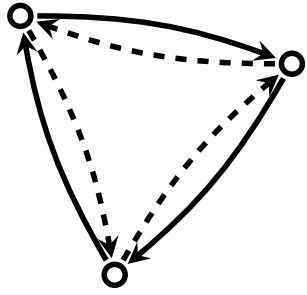

\emph{Heterogeneity in the coupling function is necessary} to realize the Guckenheimer--Holmes heteroclinic cycle. 
In the remainder of \Cref{sec:gh-directed}, we restrict ourselves to the case of homogeneous coupling in every order and explore how higher-order interactions can be used to break the symmetry in the inputs so that the Guckenheimer--Holmes cycle may arise.

%%%
\subsection{Only True \texorpdfstring{$2$}{2}-to-\texorpdfstring{$1$}{1} Connections}
\label{subsubsec:gh-2-to-1-only}
First, we note that a hypergraph can have `too few' hyperedges to induce heterogeneity in the inputs necessary for the Gucken\-heimer--Holmes cycle. In~\eqref{eq:gh_cubic}, each node receives inputs from both other nodes. One might consider this a triplet interaction so that the obvious choice would be to consider dynamics on the hypergraph with edges
\begin{equation}
    \label{eq:gh-H2}
    \begin{split}
        \HHo{3} &= \{\he{2,3}{1}, \he{1,3}{2}, \he{1,2}{3}\}, \text{ and}\\
        \HHo{m} &= \emptyset \text{ whenever } m \ne3.
    \end{split}
\end{equation}
Note that these are `true' triplet interactions as the head an tail sets of all edges do not intersect.

\begin{thr}
    Dynamics on a hypergraph~$\Hc$ with three vertices
    that only contains effective $2$-to-$1$ hyperedges does not allow for the realization of the Guckenheimer--Holmes system \eqref{eq:gh_cubic} as a network dynamical system.
\end{thr}
\begin{proof}
    A hypergraph $\Hc$ with three vertices that only contains effective $2$-to-$1$ hyperedges is necessarily of the form \eqref{eq:gh-H2}. We may trivially generate the interaction terms $bzy_1^2+czy_2^2$ via a third order coupling function $\go{3}(z;y_1,y_2)$. However, this requires the coupling function not to be invariant under transposition of the inputs: $\go{3}(z;y_1,y_2) \not\equiv \go{3}(z;y_2,y_1)$. In particular, this would require that each (directed) hyperedge (of order $3$) is sensitive to the order of its inputs, which is not part of the framework. Here, $\go{3}$---and as a result also~$f$---is symmetric in its inputs which forces $b=c$ so that the Guckenheimer--Holmes cycle cannot arise.
\end{proof}

%%%
\subsection{Directed Edges and True \texorpdfstring{$2$}{2}-to-\texorpdfstring{$1$}{1} Connections}
\label{subsubsec:gh-edge+2-to-1}
In contrast to \Cref{subsubsec:gh-2-to-1-only}, even if hyperedges are not sensitive to the order of their inputs, we may still break the symmetry in the inputs. The straightforward way to do so is to include only one pairwise edge targeting each vertex in addition to the triplet interaction. This allows each vertex to distinguish the inputs coming from any other vertex (only triplet or triplet and edge).
Again, the head an tail sets of all hyperedges do not intersect.

\begin{thr}
    Dynamics on the hypergraph $\Hc=(\Vc,\HH)$ on three vertices with edges
    \begin{equation}
        \label{eq:gh-H3}
        \begin{split}
            \HHo{2} &= \{\he{2}{1}, \he{3}{2}, \he{1}{3}\},\\
            \HHo{3} &= \{\he{2,3}{1}, \he{1,3}{2}, \he{1,2}{3}\}, \text{ and}\\
            \HHo{m} &= \emptyset \text{ whenever } m >3,
        \end{split}
    \end{equation}
    allows for the realization of the Guckenheimer--Holmes system \eqref{eq:gh_cubic} as a network dynamical system.
\end{thr} 
\begin{proof}
    To prove this assertion, we just give possible coupling functions explicitly. Define 
    \begin{align*}
        F(z)            &= z+az^3 \\
        \go{2}(z;y_1)     &= (b-c)zy_1^2 \\
        \go{3}(z;y_1,y_2)   &= czy_1^2 + czy_2^2 = \go{3}(z;y_2,y_1).
    \end{align*}
    They realize the desired vector field as in~\eqref{eq:GH-F} governed by 
    \[ f(z,y_1,y_2) = F(z) + \go{2}(z;y_1) + \go{3}(z;y_1,y_2). \]
\end{proof}
An analogous construction is possible in the hypergraph in which the classical edges---i.e., hyperedges of order $2$---have exchanged sources for each vertex simultaneously. As this example shows that it is possible to generate the cubic Guckenheimer--Holmes system \eqref{eq:gh_cubic} if the coupling is homogeneous in every order, we will restrict to this case in the remainder of this section.

%%%
\subsection{Self-Influence Through Intersecting Head and Tail}
So far, we have always assumed that the hyperedges have disjoint heads and tails, that is, the (hyper)edges are actually of effective order two and three.
By contrast, hyperedges where the head and tail sets intersect have a lower effective order. For example, the hyperedge $\he{1,2}{1}$ is degenerate in the sense that it corresponds to an effective coupling between a pair of nodes.
Such coupling gives more flexibility to choose interaction functions to obtain a desired vector field. 
This allows to construct the Guckenheimer--Holmes system~\eqref{eq:gh_cubic} similarly to \Cref{subsubsec:gh-edge+2-to-1}. 

\begin{thr}\label{thm:Degenerate}
    The hypergraph $\Hc=(\Vc,\HH)$ on three vertices with hyperedges
    \begin{equation}
        \label{eq:gh-H4}
        \begin{split}
            \HHo{2} &= \{\he{2}{1}, \he{3}{2}, \he{1}{3}\},\\
            \HHo{3} &= \{\he{1,3}{1}, \he{1,2}{2}, \he{2,3}{3}\}, \text{ and}\\
            \HHo{m} &= \emptyset \text{ whenever } m >3,
        \end{split}
    \end{equation}
    allows for the realization of the Guckenheimer--Holmes system~\eqref{eq:gh_cubic} as a network dynamical system.
\end{thr} 
\begin{proof}
    With coupling functions
    \begin{align*}
        F(z)            &= z+(a-c)z^3 \\
        \go{2}(z;y_1)     &= bzy_1^2 \\
        \go{3}(z;y_1,y_2)   &= czy_1^2+czy_2^2 = \go{3}(z;y_2,y_1)
    \end{align*}
    the hypergraph~$\Hc$ realizes the desired governing function~$f$ as in~\eqref{eq:GH-F}, as
    \[ f(z,y_1,y_2) = F(z) + \go{2}(z;y_1) + \go{3}(z;y_1,y_2). \]
\end{proof}

Note that the symmetry of~$\go{3}$ causes an additional term~$cz^3$. This can be compensated by the internal dynamics function $F(z)$---we could have also included classical self-loop edges for this purpose.

Inspecting interaction terms in~\eqref{eq:gh_cubic}---or equivalently of the governing function~$f$ in~\eqref{eq:GH-F}---note that the monomial terms describe either self-influences or pairwise interactions.
This is also the reason why we can generate \eqref{eq:gh_cubic} using only classical edges (cf.~\Cref{subsubsec:gh-classical}) as long as there are edges of two types.
Dynamics on a hypergraph with degenerate higher-order interactions that are effectively pairwise interactions, as considered here, are a different way to generate two types of interactions necessary for the Guckenheimer--Holmes cycle.
While the example in \Cref{thm:Degenerate} has degenerate hyperedges of order three, this generalizes to degenerate hyperedges of any higher order.

%%%
\subsection{Uniform Hypergraphs}
\label{subsec:gh-directed-uniform}
The final class of dynamical systems we consider in the context of the Guckenheimer--Holmes cycle are network dynamics on $m$-uniform hypergraphs; these are commonly considered in the literature.
In particular, all hyperedges are of the same order.
Since the network will have three vertices, we only have to consider $2$-, $3$-, and $4$-uniform hypergraphs.

First, note that $2$-uniform hypergraphs are graphs and thus the question reduces to realizing the Guckenheimer--Holmes cycle in a classical network dynamical system with pairwise interactions.
As we have seen in \Cref{subsubsec:gh-classical}, this requires two different coupling functions for hyperedges of order two.
Hence, under the condition of homogeneous coupling in every order it is not possible to construct the Guckenheimer--Holmes cycle in a $2$-uniform hypergraph.

Second, we consider dynamics on $3$-uniform hypergraphs. 
If we do not allow self-influences---meaning no degenerate hyperedges---the only $3$-uniform hypergraph on three vertices is the one considered in \Cref{subsubsec:gh-2-to-1-only}. 
As we have shown, dynamics on such hypergraphs cannot realize the Guckenheimer--Holmes system~\eqref{eq:gh_cubic}.
If we allow for $3$-uniform hypergraphs with self-couplings, the resulting class of hypergraphs on three vertices is much larger. Dynamics on many such hypergraphs allow for the generation of the Guckenheimer--Holmes system~\eqref{eq:gh_cubic} as summarized in the following statement.

\begin{thr}
    \label{thr:gh-3-uniform}
    The Guckenheimer--Holmes system \eqref{eq:gh_cubic} can be realized as a network dynamical systems on $3$-uniform hypergraphs. The admissible hypergraphs can be organized into $20$ different categories.
\end{thr}
\noindent
Parts of the proof are simply technical and not very enlightening. For this reason, we present only the general idea with an example here. The remainder of the proof, as well as the list of suitable categories together with example hypergraphs can be found in \Cref{sec:app-gh-3-uniform}.
\begin{proof}
    We allow for all hyperedges of the forms $\he{i,j}{k}, \he{i,j}{k,l}$, and $\he{i,j}{1,2,3}$ where $i,j,k,l\in\{1,2,3\}$ with $i\ne j$ as well as $k\ne l$. These are all hyperedges of order three. In particular, the dynamics on any vertex is governed by the sum of the internal function~$F$ and multiple instances of~$\go{3}$ whose inputs are determined by the hyperedges that have the vertex in their head. 
    To generate the cubic Guckenheimer--Holmes system, this sum has to equal~$f$ in~\eqref{eq:GH-F} for any vertex. 
    
    We may focus on one arbitrary vertex, say vertex~$k$, for the construction. 
    We denote the state variable of the corresponding node in the network dynamical system by $z\in\R$ and those corresponding to the two neighbors by $y_1, y_2 \in\R$. To generate~$f$ in~\eqref{eq:GH-F}, the intrinsic dynamics~$F$ and coupling function~$\go{3}$ have to be polynomial in their arguments. 
    In fact, $F$~has to contain the monomials~$z$ and~$z^3$, while $\go{3}(z;y_1,y_2)$ has to contain the monomials~$zy_1^2$ and~$zy_2^2$ to be able to generate the corresponding terms in~$f$. 
    They are thus of the form
    \begin{align*}
        F(z)            &= \alpha_1z + \alpha_2z^3 + P(z), \\
        \go{3}(z;y_1,y_2)   &= \beta zy_1^2+ \beta zy_2^2 +P'(z;y_1,y_2),
    \end{align*}
    where~$P$ and~$P'$ collect all terms such that~$P$ does not contain the monomials~$z$ and $z^3$, while~$P'$ does not contain the monomials $zy_1^2$ and~$zy_2^2$. Without loss of generality, we also assume that $P'$ does not contain the monomials~$z$ and $z^3$ either, as these would merely cause shifted values for $\alpha_1$ and $\alpha_2$ in the argumentation and results that follow, cf. \eqref{eq:gh-directed-uniform} and below. Note that there is only one coefficient $\beta\in\R$ in~$\go{3}$ due to the symmetry in~$y_1$ and~$y_2$.
    
    Let~$E_0$ denote the set of true $2$-to-$1$ hyperedges whose head contains the chosen vertex~$k$. Moreover, let~$E_1$ and~$E_2$ denote the sets of degenerate hyperedges whose head and tail contain~$k$ and whose tail contains one of the other neighbors resulting in a $y_1$- or a $y_2$-influence, respectively. 
    Then the dynamics of the vertex~$k$ in focus is governed by 
    \begin{equation}
        \label{eq:gh-directed-uniform}
        \begin{split}
            &F(z) 
            + \sum_{e\in E_0} \go{3}(z;y_1,y_2) 
            + \sum_{e\in E_1} \go{3}(z;z,y_1) 
            + \sum_{e\in E_2} \go{3}(z;z,y_2) \\
            &\quad = F(z) 
            + \Pi \go{3}(z;y_1,y_2) 
            + \Phi \go{3}(z;z,y_1) 
            + \Psi \go{3}(z;z,y_2) \\
            &\quad= \alpha_1z+\alpha_2z^3 + \Pi\beta (zy_1^2+zy_2^2) + \Phi\beta (z^3+zy_1^2) % \\ 
%            &\phantom{=\alpha_1X} 
            + \Psi\beta (z^3+zy_2^2)
            + Q(z,y_1,y_2) \\
            &\quad = \alpha_1z
            + (\alpha_2 + \beta (\Phi+\Psi))z^3 
            + \beta(\Pi+\Phi)zy_1^2 %\\ 
%            &\phantom{=\alpha_1X} 
            + \beta(\Pi+\Psi)zy_2^2
            + Q(z,y_1,y_2),
        \end{split}
    \end{equation}
    where~$\Pi=|E_0|, \Phi=|E_1|, \Psi=|E_2|$ and~$Q(z,y_1,y_2)$ is a polynomial that does not contain the monomials~$z$,~$z^3$, $zy_1^2$ and $zy_2^2$. 
    In particular, $\beta(\Pi+\Phi)zy_1^2 + \beta(\Pi+\Psi)zy_2^2$ cannot be equal to $bzy_1^2+czy_2^2$ for \emph{arbitrary} $b, c\in\R$, since $\Pi, \Phi, \Psi$ are integers. 
    
    For suitable $\Pi, \Phi, \Psi$, we can however chose suitable values for $\alpha_1, \alpha_2, \beta \in\R$ to satisfy the sufficient conditions for the emergence of the Gucken\-heimer--Holmes cycle (cf.~below \eqref{eq:gh_cubic})---in fact, we always need $\alpha_1=1$ and thus abbreviate $\alpha=\alpha_2$. 
    For example, assume $\Pi=0, \Phi=1, \Psi=2$. Then we choose
    \begin{align*}
        F(z)            &= z+\frac{2}{5}z^3, \\
        \go{3}(z;y_1,y_2)   &= -\frac{7}{30} zy_1^2-\frac{7}{30}zy_2^2,
    \end{align*}
    that is, $\alpha=\frac{2}{5}$ and $\beta=-\frac{7}{30}$, and $P\equiv0, P'\equiv0$. Then, 
    \begin{align*}
        &F(z) 
            + \sum_{e\in E_1} \go{3}(z;y_1,y_2) 
            + \sum_{e\in E_2} \go{3}(z;z,y_1) 
            + \sum_{e\in E_3} \go{3}(z;z,y_2) \\
        &\quad= z-\frac{3}{10}z^3-\frac{7}{30}zy_1^2-\frac{7}{15}zy_2^2.
    \end{align*}
    This equals $f(z;y_1,y_2)$ in \eqref{eq:GH-F} with $a=-\frac{3}{10}, b=-\frac{7}{30}, c=-\frac{7}{15}$. 
    In particular, these parameters satisfy $a+b+c=-1, -\frac{1}{3}<a<0, c<a<b<0$ so that the Guckenheimer--Holmes cycle exists for this governing function. 
    
    It remains to be shown that there is indeed a configuration of hyperedges, such that $\Pi=0, \Phi=1, \Psi=2$ for all three vertices. In fact, choosing
    \begin{align*}
        \HH = \HHo{3} = \{ &\he{1,3}{1}, \he{1,2}{2}, \he{2,3}{3}, \\
            & \he{1,3}{1,3}, \he{1,2}{1,2}, \he{2,3}{2,3} \},
    \end{align*}
    we obtain a hypergraph in which no vertex is in the head of a true $2$-to-$1$ hyperedge, and each vertex receives one degenerate input from its next neighbor and two degenerate inputs from its second neighbor. In total, there are only~$20$ configurations of~$\Pi, \Phi, \Psi$ that can be realized by a $3$-uniform hypergraph and that allow for the realization of the Guckenheimer--Holmes cycle. A full list together with example hypergraphs can be found in \Cref{tab:ap-gh-uniform} in \Cref{sec:app-gh-3-uniform}.
\end{proof}

Finally, we turn to $4$-uniform hypergraphs.

\begin{thr}
    The Guckenheimer--Holmes system \eqref{eq:gh_cubic} cannot be realized as a network dynamical system on a $4$-uniform hypergraph.
\end{thr}
\begin{proof}
    All hyperedges in a $4$-uniform hypergraph on three vertices are of the form $\he{1,2,3}{k}, \he{1,2,3}{k,l}$, or $\he{1,2,3}{1,2,3}$. 
    Thus, the state of the $k$th vertex evolves according to
    \[ \dot{x}_k=F(x_k)+ \sum_{\substack{e\in\HH \\ k\in H(e)}} \go{4}(x_k;x_1,x_2,x_3). \]
    Due to the symmetry properties of~$\go{4}$, the right hand side is invariant under exchanging~$x_i$ and~$x_j$ for $i,j\ne k$. As we established before, this cannot realize the cubic Guckenheimer--Holmes system~\eqref{eq:gh_cubic}.
\end{proof}

%%%
\subsection{Structural Stability and Generic Dynamics on Hypergraphs}
So far, we have focused on explicitly realizing the cubic Gucknheimer-Holmes system as dynamical systems on different hypergraphs.
However, the classical result in~\cite{Guckenheimer.1988} is much stronger: 
It states that the heteroclinic cycle not only exists for the particular cubic vector field given but also for any perturbation of the vector field that preserves the symmetry properties.
Specifically, there exists an open subset of the set of all vector fields of the form
\begin{equation}
    \label{eq:gh-sym1}
    \begin{pmatrix}
        f(x_1,x_2,x_3) \\
        f(x_2,x_3,x_1) \\
        f(x_3,x_1,x_2)
    \end{pmatrix},
\end{equation}
with~$f$ satisfying the symmetry condition
\begin{equation}
    \label{eq:gh-sym2}
    f(x_1,x_2,x_3)=-f(-x_1,x_2,x_3)=f(x_1,-x_2,x_3)=f(x_1,x_2,-x_3)
\end{equation}
for which the Guckenheimer--Holmes cycle exists.

This has several consequences that relate to structural stability.
First, for any of the given hypergraphs above where the Guckenheimer--Holmes cycle could be realized it is actually \emph{stable under perturbations of the coupling functions} (as long as it preserves the symmetry properties). For example, perturbations of the coupling function by weak higher order monomials does not destroy the cycle.
Second, if we allow for distinct coupling functions for each edge, then we have persistence of the heteroclinic cycle under \emph{structural perturbations of the hypergraph} (for example, adding an edge with a coupling function that is uniformly bounded and small) as long as the symmetry properties of the vector field are preserved.
Third, one can get existence for the class of hypergraphs that contain all possible hyperedges as shown in the following statement. 

\begin{thr}
    The Guckenheimer--Holmes {cycle} emerges in network dynamical systems on a hypergraph on three vertices with \emph{all} possible hyperedges.
\end{thr}
\begin{proof}
    A generic hypergraph dynamical system on three vertices with homogeneous coupling in every order is a system in which \emph{all} hyperedges are present and the coupling functions are generic in the sense that their formal series expansions~\eqref{eq:expansion} contain generic coefficients. 
    Since all hyperedges are present and there is only one coupling function per hyperedge order, the system is necessarily equivariant under arbitrary permutations of the three vertices.
    In particular, it is a special case of~\eqref{eq:gh-sym1}, 
    which are precisely the ones that are equivariant under cyclic permutations. 
    If we additionally assume that the coupling functions satisfy
    \begin{align*}
        F(z)                &= -F(-z) \\
        \go{2}(z;y_1)         &= -\go{2}(-z;y_1)= \go{2}(z;-y_1) \\
        \go{3}(z;y_1,y_2)       &= -\go{3}(-z;y_1,y_2) = \go{3}(z;-y_1,y_2) = \go{3}(z;y_1,-y_2) \\
        \go{4}(z;y_1,y_2,y_3)     &= -\go{4}(-z;y_1,y_2,y_3) = \go{4}(z;-y_1,y_2,y_3) \\
                            &= 
                            \go{4}(z;y_1,-y_2,y_3) = \go{4}(z;y_1,y_2,-y_3)
    \end{align*}
    any sum of them also satisfies~\eqref{eq:gh-sym2}. 
    This holds even if the internal variable~$z$ is also one of the coupling variables. In particular, we assume that the formal series expansions~\eqref{eq:expansion} contain only terms that are simultaneously of odd degrees in the internal variable and of even degrees in the input variables. 
    Note that for~$\go{4}$ we have $y_j=z$ for some~$j$, as there are only three different variables and no variable can be a coupling variable twice.
    
    Thus, under these assumptions a small enough perturbation of any of the cubic Guckenheimer--Holmes systems we were able to construct on a hypergraph by a generic hypergraph dynamical system preserves the heteroclinic cycle. In fact, this implies that also the hypergraph with all hyperedges supports the Guckenheimer--Holmes cycle.
\end{proof}

\begin{remk}
    Note that the same argumentation yields the same result for any hypergraph that guarantees the cyclic equivariance~\eqref{eq:gh-sym1}. 
    Heuristically speaking, this is satisfied if and only if the different types of couplings are distributed cyclically over the three vertices with the types being
    \begin{itemize}
        \item classical pairwise edges,
        \item true $2$-to-$1$ couplings,
        \item degenerate hyperedges of order three that realize a $1$-to-$1$ coupling from the right neighbor,
        \item degenerate hyperedges of order three that realize a $1$-to-$1$ coupling from the left neighbor,
        \item and inputs by all three vertices (here the precise structure is not important, e.g., $\he{1,2,3}{1,2,3}$ yields the same terms in the equations of motions as the combination of~$\he{1,2,3}{1}$ and~$\he{1,2,3}{2,3}$.
    \end{itemize}
\end{remk}

%%%%%%
\section{The Field Cycle for Dynamics on Directed Hypergraphs}
\label{sec:oscar-directed}
\noindent
Now we turn to the Field cycle as an example of the more general construction to obtain dynamical systems with prescribed heteroclinic structures. 
Contrary to the previous section, the field cycle has two saddle equilibria in the subspace~$\Delta$ corresponding to full synchrony that are connected by heteroclinic trajectories that lie in two different subspaces that correspond to partially synchronous states. 
As we have seen in \Cref{subsec:oscar-undirected}, undirected hyperedges preserve `too many symmetries' in the equations for the construction to work. 
Similar to the Guckenheimer--Holmes cycle, this is no longer the case if directed hyperedges are considered. 
However, the different setup leads to some significant differences between realizing the Field cycle compared to the Guckenheimer--Holmes cycle.
For example, Field's construction cannot work in $m$-uniform hypergraphs. 
Before investigating the construction in more detail below, we can make several straightforward observations similar to the Guckenheimer--Holmes cycle in \Cref{sec:gh-directed}.

The Field cycle as an example of the construction in~\cite{Aguiar.2011,Field2017,Weinberger.2018} is done for a classical network with two types of couplings (see \Cref{fig:oscar_net}), i.e., two different coupling functions for edges. 
In that regard, it is very similar to the Guckenheimer--Holmes system and many of the general observations made in \Cref{sec:gh-directed} can be made here just as well. 
We can obviously follow the same construction if the hypergraph is a classical network without the assumption of homogeneous coupling in every order---in fact, it was shown in \cite{Field.2015} that the construction can be performed for dyadic networks with additive input structure. 
The same can still be done if we restrict to homogeneous coupling in every order by replacing one type of pairwise couplings with degenerate $2$-to-$1$ hyperedges. 
These allow us to distinguish two different types of pairwise coupling through $\go{2}$ and $\go{3}$. 
On the other hand, if we include hyperedges in a symmetric manner---e.g., all edges, all true $2$-to-$1$ hyperedges etc.---the construction is not possible.

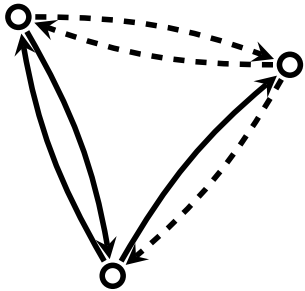
\begin{figure}
    \centering
    \resizebox{!}{.3\textwidth}{
        \begin{tikzpicture}[
%			background rectangle/.style={
%				fill=grey!10
%			},
%			show background rectangle,
%	->,
%	>=stealth',
%	shorten >=1pt,
%	shorten <=1pt,
	%		auto,
	%		node distance=5cm,
	square/.style = {
		regular polygon,
		regular polygon sides=4
	},
	main node/.style={
		%			node distance=2cm,
		line width=1.5pt, 
		circle,
		draw,
		font=\sffamily,
		inner sep=2pt,
		fill=white
	},
	second node/.style={
		%			node distance=2cm,
		line width=1.5pt, 
		square, 
		draw, 
		font=\sffamily,
		rounded corners, 
		inner sep=2pt
	},
	edge/.style={
		-stealth,
		shorten >=1pt,
		shorten <=1pt,
		line width=1.5pt
	},
	hyperedge/.style={
		Round Cap-{Triangle[length=3mm, width=2mm]},
		line width=5pt
	}
	]
	\node[rotate=50, regular polygon, regular polygon sides=3, minimum width=3cm] (tr) at (0,0) {};
	
	\node[main node] (1) at (tr.corner 1) {};
	\node[main node] (2) at (tr.corner 2) {};
	\node[main node] (3) at (tr.corner 3) {};
	
	%edges
	\path[line width=1.5pt]
	(2) edge [edge, bend left=10] (1)
	(3) edge [edge, bend left=10, dashed] (1)
	
	(1) edge [edge, bend left=10] (2)
	(3) edge [edge, bend left=10, dashed] (2)
	
	(1) edge [edge, bend left=10, dashed] (3)
	(2) edge [edge, bend left=10] (3)
	;
\end{tikzpicture}%
	}%
    \caption{The $3$-vertex graph with two edge types that supports the construction in \cite{Weinberger.2018}.}
    \label{fig:oscar_net}
\end{figure}

%%%
\subsection{Preliminary Observations}
Recall that the construction of the Field cycle crucially depends on the (partial) synchrony subspaces that are dynamically invariant for the given network independent of the specific governing functions. 
In a network with three vertices there are four synchrony subspaces 
$\Delta=\{x_1=x_2=x_3\}$, 
$S_3=\{x_1=x_2\}$, 
$S_2=\{x_1=x_3\}$, 
$S_1=\{x_2=x_3\}$. 
If the intrinsic node dynamics are one-dimensional, as we assume here, the first subspace is one-dimensional while all other ones are two-dimensional. 
We require that the fully synchronous subspace is robust, which requires some form of homogeneity in the hypergraph and the equations of motion as summarized in the following statement.

\begin{lem}
    \label{lem:oscar-fullsynch}
    Consider system~\eqref{eq:hypernetdyn} as a network dynamical system on a hypergraph with three vertices. Assume that the coupling is homogeneous in every order. 
    The fully synchronous subspace~$\Delta$ is dynamically invariant for each choice of functions $F, \go{2}, \go{3}, \dotsc$ if and only if the number of hyperedges of order~$m$ targeting a given vertex is the same for each vertex.
\end{lem}
\begin{proof}
    Assume dynamical invariance of $\Delta = \{x_1 = x_2 = x_3\}$. Then for any point $\x=(z,z,z)^\tr\in\Delta$, the right hand sides of~\eqref{eq:hypernetdyn} must be equal for each~$k$. 
    Write $N(m; k):=\# \left\{e\in \HHo{m} \mid k\in H(e) \right\}$ for the number of hyperedges of order~$m$ whose heads contain vertex~$k$.
    Substituting~$\x$ into the system we have
    \begin{equation}
        \label{eq:fullsynch}
        \begin{split}
        \dot z &= F(z) + \sum_{\substack{e\in \HHo{2} \\ k\in H(e)}} \go{2} (z; z) + \sum_{\substack{e\in \HHo{3} \\ k\in H(e)}} \go{3} (z; z, z) + \dotsb\\
            & = F(z) + N(2; k)\go{2} (z; z)
            +N(3; k)\go{3} (z; z, z) + \dotsb.
            \end{split}
    \end{equation}
    This expression is independent of~$k$ if and only if~$N(m; k)$ are independent of~$k$, which proves the lemma.
\end{proof}
\begin{remk}
    The result is not restricted to the case $N=3$. The same proof works for arbitrary values of $N$.
\end{remk}
\begin{remk}
    We will see below, that it is possible to construct the Field cycle for hypergraphs without degenerate couplings and homogeneous coupling to every order (cf. \Cref{subsec:oscar-nondegenerate}). We will make this assumption in the remainder of this section.
\end{remk}

The goal of the construction is to generate hetero\-clinic connections in partially synchronous spaces between two hyperbolic equilibria in the fully synchronous subspace~$\Delta$. 
In order for such a heteroclinic connection to be possible, equilibria in~$\Delta$ need both a stable and an unstable direction outside of the fully synchronous subspace. 
In particular, the construction relies on the existence of two invariant partially synchronous subspaces in addition to the fully synchronous subspace. 
We state the following necessary condition.
\begin{lem}
    \label{lem:oscar-partsynch}
    There are no local obstructions to the construction of the Field cycle if and only if in addition to the fully synchronous subspace there are precisely two partially synchronous dynamically invariant subspaces.
\end{lem}
\begin{proof}
    Since the construction relies on the presence of two partially synchronous dynamically invariant subspaces, it is clear that it is obstructed when the network structure allows for only one of the partially synchronous subspaces. 
    On the other hand, whenever the network allows for all three partially synchronous subspaces, the procedure always fails due to a double eigenvalue of the linearization at a fully synchronous point. 
    In fact, any $3\times3$ matrix leaving $\Delta, S_3, S_2, S_1$ invariant is automatically of the form
    \begin{equation*}
        \begin{pmatrix}
            \alpha   & \beta         & \gamma \\
            \delta   & \alpha + \beta - \delta & \gamma \\
            \delta   & \beta         & \alpha + \gamma - \delta
        \end{pmatrix}.
    \end{equation*}
    This matrix has an eigenvalue $\alpha+\beta+\gamma$ with eigenvector $(1,1,1)^\tr$ as well as a double eigenvalue $\alpha-\delta$ with eigenvectors $(1,-\frac{\delta}{\beta},0)^\tr,(1,0,-\frac{\delta}{\beta})^\tr$. 
    In particular, no steady state in the fully synchronous subspaces can have a stable direction in one partial synchrony space and an unstable direction in another. 
    This prevents the heteroclinic cycle to be realized. 

    Any $3\times3$ matrix leaving~$\Delta, S_3, S_2$ invariant but not~$S_1$ is of the form
    \begin{equation*}
        \begin{pmatrix}
            \alpha   & \beta         & \gamma \\
            \delta   & \alpha + \beta - \delta & \gamma \\
            \epsilon   & \beta         & \alpha + \gamma - \epsilon
        \end{pmatrix}.
    \end{equation*}
    This matrix has eigenvalues $\alpha+\beta+\gamma, \alpha-\epsilon, \alpha-\delta$ with corresponding eigenvectors $(1,1,1)^\tr, (1,1, -\frac{\beta+\epsilon}{\gamma})^\tr,(1,-\frac{\gamma+\delta}{\beta},1)^\tr$ which poses no obstructions to the realization. Analogous observations can be made for the other two combinations of the two-dimensional subspaces.
\end{proof}
\noindent
This yields a necessary condition as to whether the construction of the Field cycle works that is easy to check: 
Does the network structure allow for the fully synchronous plus exactly two partially synchronous subspaces to be dynamically invariant independent of the governing functions?

\begin{remk}
    It can readily be seen that the dynamical invariance of a (partial) synchrony subspace is independent of the coupling functions being nodespecific or nodeunspecific. Consider two coupling functions $\go{m}(x_i;\cdot)$ and $\go{m}(x_j;\cdot)$ that target vertices $i$ and $j$ respectively. If, in a (partial) synchrony subspace to be checked, $i$ and $j$ are synchronous, both functions receive the same first argument. The same is true, if both functions depend trivially on their first argument.
\end{remk}

\begin{remk}
    The question which (partial) synchrony subspaces are dynamically invariant independent of the specific governing functions boils down to the combinatorial problem of finding so-called \emph{balanced partitions} in the network. These have famously been introduced and algebraically studied in \cite{Field.2004,Golubitsky.2006,DeVille.2015,Nijholt.2020} for coupled cell networks. 
    More recently, first advances have been made to generalise balanced partitions to network dynamical systems with higher-order interactions~ \cite{Aguiar2020,Salova.2021b,Nijholt.2022c,vonderGracht.2023}. For instance in order for vertices~$1$ and~$2$ to synchronize, the partition $\{1,2\}, \{3\}$ needs to be balanced, meaning that any vertex in an element of the partition needs to receive the same `kind' of inputs from any element of the partition as any other vertex in that same element. So assume there is a hyperedge $\he{2,3}{1}$. For the partition to be balanced, vertex~$2$ needs to receive a hyperedge with one input from the same element in the partition and one from outside. This leaves two options $\he{1,3}{2}$ and $\he{2,3}{2}$. In a similar manner, one may then construct networks with hyperedges of order $3$ that have precisely two balanced partitions.
    In order to keep the presentation as straightforward as possible, we will not use this method in the remainder of this section. Instead, we use elementary combinatorics to enumerate possible hypergraphs and check the corresponding equations of motions for invariant (partial) synchrony subspaces.
\end{remk}

\begin{remk}
    Finally, we remark that the construction for the heteroclinic dynamics in the hypergraphs presented or mentioned in this section might bear subtle differences from the one in~\cite{Field.2015} (which has been specified for the three node case in~\cite{Weinberger.2018}). 
    However, it does always follow the same lines. In particular, it can always be arranged that two trajectories that connect two fully synchronous equilibria and locally coincide with their respective one-dimensional unstable manifolds do not intersect when projected to the partial synchrony subspaces identified with~$\R^2$. Hence, this kind of global obstructions to the construction cannot occur (see \Cref{subsec:oscar-nondegenerate} below for an example and~\cite{Aguiar.2011,Field.2015} for details).
\end{remk}

%%%
\subsection{No Degenerate Hyperedges}
\label{subsec:oscar-nondegenerate}
The realization of the Field cycle is possible in network dynamical system on a hypergraph that contains only non-degenerate hyperedges, assuming that the coupling is homogeneous in every order.
The reasoning is similar to \Cref{subsubsec:gh-edge+2-to-1}.
The combination of a classical edge and a true $2$-to-$1$ hyperedge allows a targeted vertex to distinguish between the inputs of its two neighbors. 
We present the details for illustration of the restrictions that have to be checked for the construction to work.

\begin{thr}
    There is a dynamical system on a hypergraph with three vertices and edges
    \begin{equation}
        \label{eq:field-H1}
        \begin{split}
            \HHo{2} &= \{\he{2}{1}, \he{1}{2}, \he{2}{3}\},\\
            \HHo{3} &= \{\he{2,3}{1}, \he{1,3}{2}, \he{1,2}{3}\}, \text{ and}\\
            \HHo{m} &= \emptyset \text{ whenever } m >3
        \end{split}
    \end{equation}
    that realizes the Field cycle.
\end{thr}
\begin{proof}
    Dynamics on the hypergraph on three vertices with edges~\eqref{eq:field-H1} evolve according to
    \begin{equation}
        \label{eq:oscar-nondegenerate}
        \begin{split}
            \dot{x}_1   &= F(x_1) + \go{2} (x_1; x_2) +  \go{3} (x_1;x_2,x_3) \\
            \dot{x}_2   &= F(x_2) + \go{2} (x_2; x_1) + \go{3} (x_2;x_1,x_3) \\
            \dot{x}_3   &= F(x_3) + \go{2} (x_3; x_2) + \go{3} (x_3;x_1,x_2).
        \end{split}
    \end{equation}
    This system has dynamically invariant subspaces~$\Delta$, $S_3$, and~$S_2$; however, $S_1$~is not dynamically invariant. 
    Thus, there are no local obstructions. 
    
    We now look at the local situation in more detail. The linearization of~\eqref{eq:oscar-nondegenerate} at a fully synchronous equilibrium $\mathbf{x}=(z,z,z)^\mathsf{T}\in\Delta$ is of the form
    \begin{equation*}
        \begin{pmatrix}
            \alpha + \beta + \gamma & \delta + \epsilon & \epsilon \\
            \delta + \epsilon & \alpha + \beta + \gamma & \epsilon \\
            \epsilon & \delta + \epsilon & \alpha + \beta + \gamma \\
        \end{pmatrix},
    \end{equation*}
    where
    \begin{subequations}\label{eq:ABGDE}
    \begin{align}
        \alpha &= F'(z), \\
        \beta &= \partial_1\go{2}(z;z), \\
        \gamma &= \partial_1 \go{3}(z;z,z), \\
        \delta &= \partial_2 \go{2}(z;z), \\
        \epsilon &= \partial_2 \go{3}(z;z,z) =  \partial_3 \go{2}(z;z,z).
    \end{align}
    \end{subequations}
    This matrix has eigenvalues and corresponding eigenvectors
    \begin{align*}
            \lambda_1&=\alpha+\beta+\gamma+\delta+2\epsilon,    & v_1 &= (1,1,1)^\tr, \\
            \lambda_2&=\alpha+\beta+\gamma-\delta-\epsilon,  & v_2 &= \left(1,1,-\frac{\delta+2\epsilon}{\delta+       \epsilon}\right)^\tr\\
            \lambda_3&=\alpha+\beta+\gamma-\epsilon, &     
            v_3 &= \left(1,-\frac{\delta+2\epsilon}{\epsilon},1\right)^\tr.
    \end{align*}
    For two different fully synchronous equilibria~$\mathbf{p}$, $\mathbf{q}$ it can then be arranged that they each have a one-dimensional unstable manifold in different partially synchronous subspaces by choosing suitable parameter values $\alpha, \dotsc, \epsilon\in\R$. Furthermore, projecting~$S_3$ and~$S_2$ to~$\R^2$, it can be arranged that the slopes of the corresponding eigenlines have opposite signs---this requires $(\delta+\epsilon)\epsilon<0$ and both factors to have different signs at both equilibria. Then two trajectories from~$\mathbf{p}$ to~$\mathbf{q}$ and vice versa that coincide with the corresponding unstable manifolds locally can be assumed not to intersect when projected onto~$\R^2$ (cf.~Figure~18 in \cite{Aguiar.2011}). Since these projections are related by the network structure, there are no global obstructions to the simultaneous existence of both connecting trajectories. Hence, one can choose suitable coupling functions realizing the Field cycle (see~\cite{Aguiar.2011,Field.2015} and \Cref{rem:one-sided} below for additional details).
\end{proof}

\begin{remk}
	\label{rem:one-sided}
	In network dynamical systems, projections onto different (dynamically invariant) subspaces are typically related by the network structure. For example, it may happen that the function governing the dynamics of one coordinate in one projection equals that of a different coordinate in the other projected system. If the local properties furthermore force connecting trajectories to intersect when projected onto the identified lower dimensional subspaces (cf.~Figure~18 in \cite{Aguiar.2011}) one cannot construct connecting trajectories independently of each other. In such a situation the construction of the Field cycle can be prohibited entirely by these global features, we say it is \emph{globally obstructed}. In the previous proof however we may arrange for the connecting trajectory from~$\mathbf{p}$ to~$\mathbf{q}$ to lie on one side of the (dynamically invariant) diagonal in the projection to~$\R^2$ and for the connecting trajectory from~$\mathbf{q}$ to~$\mathbf{p}$ to lie on the other side. 
 Thus, no global obstructions occur. This approach realizes the heteroclinic connections without additional control over the other half of respective stable and unstable manifolds.
\end{remk}

The observation in \Cref{eq:oscar-nondegenerate} does not change if the coupling is nodeunspecific. Note these network dynamics on a hypergraph are different from the one in \Cref{subsubsec:gh-edge+2-to-1} which allowed for the Guckenheimer--Holmes cycle. A similar argument to this section can be applied to dynamics on a hypergraph with hyperedges of order four, as these are necessarily degenerate and model $2$-to-$1$ couplings.

%%%
\subsection{Uniform Hypergraphs}
While the Guckenheimer--Holmes cycle may be realized in network dynamical systems on $m$-uniform hypergraphs (cf.~\Cref{subsec:gh-directed-uniform}), they do not provide enough degrees of freedom to realize the Field cycle. 
\begin{thr}
    The Field cycle cannot be realized in a network dynamical system on an $m$-uniform hypergraph on three vertices with homogeneous coupling.
\end{thr}
\begin{proof}
    The following observations can be made immediately: $2$-uniform hypergraphs are classical networks which we have discussed in the beginning of \Cref{sec:oscar-directed}. 
    On the other hand, $3$-uniform hypergraphs without non-degenerate hyperedges as well as all $4$-uniform hypergraphs are symmetric in the inputs and therefore they do not allow for the construction (see \Cref{subsec:gh-directed-uniform}).

    The situation for $3$-uniform hypergraphs with degenerate hyperedges is more subtle. Without specifying the hyperedges that are present in the hypergraph, the equations of motion are
    \begin{equation}
        \label{eq:oscar-directed-uniform}
        \begin{split}
            \dot{x}_1   &= F(x_1) + \Pi_1\go{3}(x_1;x_2,x_3) + \Phi_1\go{3} (x_1;x_1,x_2) + \Psi_1\go{3}(x_1;x_1,x_3) \\
            \dot{x}_2   &= F(x_2) + \Pi_2\go{3}(x_2;x_1,x_3) + \Phi_2\go{3} (x_2;x_2,x_3) + \Psi_2\go{3}(x_2;x_1,x_2) \\
            \dot{x}_3   &= F(x_3) + \Pi_3\go{3}(x_3;x_1,x_2) + \Phi_3\go{3} (x_3;x_1,x_3) + \Psi_3\go{3}(x_3;x_2,x_3).
        \end{split}
    \end{equation}
    Here, the (nonnegative) integers $\Pi_k, \Phi_k, \Psi_k$ for $k=1,2,3$ count the number of true hyperedge inputs as well as degenerate inputs from the left and right neighbors respectively that vertex~$k$ receives. 
    Due to \Cref{lem:oscar-fullsynch}, robustness of the fully synchronous subspace~$\Delta$ requires $\Pi_1+\Phi_1+\Psi_1=\Pi_2+\Phi_2+\Psi_2=\Pi_3+\Phi_3+\Psi_3 =: \Xi$. 
    Furthermore, following \Cref{lem:oscar-partsynch}, the construction requires robustness of precisely two partial synchrony subspaces. 
    System~\eqref{eq:oscar-directed-uniform} is symmetric in $x_1,x_2,x_3$ and, without loss of generality, we assume~$S_3$ and~$S_2$ to be dynamically invariant. 
    Substituting these assumptions into~\eqref{eq:oscar-directed-uniform} we additionally obtain $\Xi-\Phi_1=\Pi_2+\Phi_2$ and $\Xi-\Phi_3=\Pi_1+\Phi_1$.
    
    With the assumptions and $\alpha,\beta, \dotsc, \epsilon$ as in~\eqref{eq:ABGDE}, the linearization of the right hand side of~\eqref{eq:oscar-directed-uniform} at a fully synchronous equilibrium has the form
    \begin{equation*}
    	\begin{pmatrix}
    		\alpha + \Xi\beta + (\Xi-\Pi_1)\gamma & (\Xi+\Pi_1-\Pi_2-\Phi_2)\gamma & (\Pi_2+\Phi_2)\gamma \\
    		(\Xi-\Phi_2)\gamma & \alpha + \Xi\beta + (\Xi-\Pi_2)\gamma & (\Xi+\Phi_2)\gamma \\
    		(-\Pi_1+\Pi_2+\Phi_2+\Pi_3)\gamma & (\Xi+\Pi_1-\Pi_2-\Phi_2)\gamma & \alpha + \Xi\beta + (\Xi-\Pi_3)\gamma \\
    	\end{pmatrix},
    \end{equation*}
    This matrix has eigenvalues with corresponding eigenvectors
 \begin{align*}
        \lambda_1&=\alpha+\Xi\beta+2\Xi\gamma,
            & v_1 &= (1,1,1)^\tr, \\
            \lambda_2&=\alpha+\Xi\beta+(\Xi-\Pi_2-\Phi_2-\Pi_3)\gamma,  & 
            v_2 &= \left(1,1,-\frac{\Xi+\Pi_3}{\Pi_2+\Phi_2}\right)^\tr\\            
            \lambda_3&=\alpha+\Xi\beta+(-\Pi_1+\Phi_2)\gamma, &     
            v_3 &= \left(1,-\frac{\Xi+\Pi_2}{M+\Pi_1-\Pi_2-\Phi_2},1\right)^\tr.
    \end{align*}    
    One readily observes
    \begin{align*}
    	\lambda_2&=\lambda_1 + (-\Xi-\Pi_2-\Phi_2-\Pi_3)\gamma, \\
    	\lambda_3&=\lambda_1 + (-2\Xi-\Pi_1+\Phi_2)\gamma.
    \end{align*}
    In both equations, the integer coefficient of~$\gamma$ is negative---recall that $\Phi_2<\Xi$. 
    
    Consider the case that $\lambda_1$ and $\gamma$ have opposing signs, $\lambda_1\gamma\le0$. Then, we either have $\lambda_2,\lambda_3\le0$ or $\lambda_2,\lambda_3\ge0$, i.e., the two eigenvalues have the same sign. In this situation, the fully synchronous equilibrium is not a saddle with a partially synchronous stable and unstable eigendirection which prohibits the construction. On the other hand, if~$\lambda_1$ and~$\gamma$ have the same sign, $\lambda_1\gamma\ge0$, either~$\lambda_2$ or~$\lambda_3$ can be arranged to be positive (or even both). However, note that
    \[ \lambda_2-\lambda_3 = (\Xi+\Pi_1-\Pi_2-2\Phi_2-\Pi_3)\gamma. \]
    The sign of the integer coefficient $(\Xi+\Pi_1-\Pi_2-2\Phi_2-\Pi_3)$ determines which of the two eigenvalues is larger. This coefficient, however, is fixed for a given hypernetwork and fully determined by its hyperedges. In particular, if precisely one of the two eigenvalues is arranged to be positive it can only be either~$\lambda_2$ or~$\lambda_3$ independent of the precise value of $\gamma$. As a result, any saddle equilibrium with one-dimensional unstable subspace has this unstable subspace contained in the same partial synchrony subspace. This is a \emph{global} obstruction to the realization of the Field cycle which requires the unstable directions of two fully synchronous saddle equilibria to be contained in the opposite partial synchrony subspaces (cf.~\Cref{subsubsec:fields-cycle}). 
    
    Finally, note that the observations do not change in the case of nodeunspecific coupling. In fact, nodeunspecific coupling does not alter the derivation of the necessary assumptions on the integer coefficients and is reflected by $\beta=0$ in the linear stability analysis.
\end{proof}

%%%%%%%

\section{Heteroclinic Cycles for Hypergraphs with \texorpdfstring{$N>3$}{N>3} Vertices}
\label{sec:LargerN}
\noindent
We focused so far on dynamics on hypergraphs with $N=3$ vertices since a three-dimensional state space is natural for the Guckenheimer--Holmes and Field heteroclinic cycles. 
Indeed, for general constructions of heteroclinic structures one is typically interested in the smallest dimension such that the heteroclinic structure can be embedded into phase space. 
For the Field cycle with $N-1=2$ connections this is a phase space of $N=3$ dimensions.
The more general construction underlying the Field cycle (cf.~Section~\ref{subsubsec:fields-cycle}) considers an $N$-dimensional phase space to realize heteroclinic structures involving $N-1$~distinct connections.

We now briefly discuss the converse question whether we can realize the Gucken\-heimer--Holmes and Field cycles in networks with $N>3$ nodes. 
In general, more nodes give more flexibility in the construction (especially in the most general case where the coupling function may depend on the edge itself), so we do not aim to be exhaustive. 
Rather, we concentrate on instructive examples that illustrate some of the opportunities and pitfalls.

%%%
\subsection{Heteroclinic Structures on Robust Synchrony Subspaces for Hypergraphs with \texorpdfstring{$N$}{N}~Vertices}

\newcommand{\rel}{\mathord{\sim}}
\newcommand{\SN}{\mathbf{S}_N}

Consider a network dynamical system on a hypergraph with vertices $\Vc=\sset{1, \dotsc, N}$. 
If~$P=\sset{P_s\subset\Vc}_s$ is a partition of~$\Vc$, that is, $\bigcup_s P_s = \Vc$ and $P_s\cap P_r = \emptyset$ if $s\neq r$, the associated synchrony subspace is $\Delta_{P} = \{x_j = x_k: j,k\in P_s\}$; these are states where the nodes in each~$P_s$ are synchronized. 
Such a synchrony subspace is a \emph{robust synchrony subspace} if it is dynamically invariant for all coupling functions.

\begin{ex}\label{ex:Symmetric}
    Symmetries induce robust synchrony subspaces. 
    Consider the complete undirected hypergraph~$\Kc$ on~$N$ vertices where all edges are present and assume that the coupling functions $G_e = G^{(m)}$ are homogeneous in each order. 
    The resulting network dynamical systems are $\SN$-symmetric where the symmetric group~$\SN$ acts on~$\Vc$ by permuting the node indices. 
    This implies that any partition of~$\Vc$ is a robust synchrony subspace~\cite{Golubitsky2002}.
\end{ex}

More generally, robust synchrony subspaces relate to \emph{quotients} of dynamical systems on hypergraphs: 
In suitable setups---see~\cite{Aguiar2020} and \cite{vonderGracht.2023} 
for complementary approaches---we have that the dynamics on robust synchrony subspaces correspond to dynamics on the corresponding \emph{quotient hypergraph} one obtains when identifying the synchronized nodes. 
For more details we refer to the references above. 
The main observation is that the dynamics on robust synchrony subspaces of network dynamical systems on hypergraphs~$\Hc$ with~$N$ vertices relate to the dynamics of quotient hypergraphs with $N'<N$ vertices.

This observation can be used to make statements about whether heteroclinic dynamics can occur in network dynamical systems on hypergraphs with $N>3$ nodes. 
We give one negative and one positive result.

\begin{cor}
    Consider dynamics on the fully symmetric hypergraph~$\Kc$ as in Example~\ref{ex:Symmetric}. No three cluster robust synchrony subspace corresponding to a partition of the vertices in three nonempty sets can contain the Guckenheimer--Holmes or the Field cycle.
\end{cor}

\begin{proof}
    The dynamics on any three-cluster robust synchrony subspace are three-dimensional.
    By assumption, the vector field corresponds to dynamics on an undirected hypergraph on $N=3$ vertices.
    Since neither the Guckenheimer--Holmes nor the Field cycle can be realized for these hypergraphs according to \Cref{thr:gh-undirected,thr:field-undirected}, they also cannot occur on robust invariant subspaces for dynamics on~$\Kc$ with $N$~vertices.
\end{proof}

The second result is immediate from the considerations in \Cref{sec:gh-directed,sec:oscar-directed}.

\begin{cor}
    Consider a network dynamical system on a hypergraph~$\Hc$ with $N>3$~vertices. 
    If the vector field on a three-cluster robust synchrony subspace reduces to a vector field corresponding to a $3$-node network where the Guckenheimer--Holmes or Field cycle can be realized then (generically) the cycle is also realized on a three-dimensional subspace of the $N$-dimensional phase space.
\end{cor}

\begin{remk}
    Since we have defined heteroclinic structures to involve hyperbolic equilibria, one has to ensure that the equilibria are hyperbolic in the full phase space. 
    That is, the eigenvalues of the linearization in the direction transverse to the robust synchrony subspace cannot have zero imaginary part.
\end{remk}

%%%
\subsection{Field's Construction for Undirected Hypergraphs on \texorpdfstring{$N$}{N}~Vertices}
Even if the smallest phase space dimension to embed the Field cycle in is $N=3$, one may be interested in generating it in a network dynamical system with $N>3$ vertices. 
The previous subsection showed, that this is possible for directed hypergraphs. 
We will now briefly discuss the undirected case.

Recall from \Cref{thr:field-undirected} that the Field cycle cannot be realized in a dynamical system on an undirected hypergraph on three vertices. 
The proof for this result has two main steps. 
First, we show that the requirements that the realization of the Field cycle imposes on the dynamics forces the coupling to be homogeneous in every order. 
Then, we show that the linearization at two fully synchronous equilibria cannot have the necessary saddle properties to allow for the emergence of reciprocal heteroclinic connections. 
This second part remains true for arbitrary~$N>3$.

\begin{thr}
    \label{thr:field-undirected-large}
    Dynamics on an undirected hypergraph on~$N$ vertices with homogeneous coupling in each order~$m$ does not allow for the realization of the Field cycle.
\end{thr}

\begin{proof}
    As before, the Field cycle consists of two fully synchronous equilibria in the dynamically invariant fully synchronous subspace and heteroclinic connections in dynamically invariant subspaces in which all but one cell are synchronous. 
    Without loss of generality, we assume that the $(N-1)$th and the $N$th cell desynchronize along the connections respectively.

    The linearization of the vector field at $\x=(z,\dotsc,z)^\tr$ is a symmetric matrix $M=(m_{i,j})\in\R^{N\times N}$ if the coupling is undirected and homogeneous in each order~$m$. 
    Furthermore, it leaves the synchrony subspaces
    \begin{align*}
        \Delta&=\{x_1=\dotsb=x_{N-1}=x_N\}\\
        S_{N-1}&=\{x_1=\dotsb=x_{N-2}=x_N\}\\
        S_N&=\{x_1=\dotsb=x_{N-1}\}
    \end{align*}
    invariant. As~$\Delta$ is one-dimensional, its spanning vector $(1,\dotsc,1)^\tr$ is an eigenvector of~$M$. This implies that~$M$ has constant row sum, which due to its symmetry also equals the sum of all columns.

    Note that
    \begin{equation*}
        S_{N-1}=\Delta\oplus \left\langle \begin{pmatrix}0\\\vdots\\0\\1\\0\end{pmatrix}\right\rangle, \qquad
        S_N=\Delta\oplus \left\langle \begin{pmatrix}0\\\vdots\\0\\0\\1\end{pmatrix}\right\rangle. 
    \end{equation*}
    Due to the invariance of these two subspaces, we obtain $M(0,\dotsc,0,1,0)^\tr\in S_{N-1}$ and $M(0,\dotsc,0,1)^\tr\in S_N$. These products, however, are the $N-1$st and the $N$th column of~$M$, respectively. Thus, we have
    \begin{equation*}
        \begin{pmatrix}
            m_{1,N-1}\\
            \vdots \\
            m_{N-2,N-1}\\
            m_{N-1,N-1}\\
            m_{N,N-1}
        \end{pmatrix} = 
        \begin{pmatrix}
            \alpha \\
            \vdots \\
            \alpha \\
            \beta \\
            \alpha
        \end{pmatrix}, \quad \text{and} \quad
        \begin{pmatrix}
            m_{1,N}\\
            \vdots \\
            m_{N-2,N}\\
            m_{N-1,N}\\
            m_{N,N}
        \end{pmatrix} = 
        \begin{pmatrix}
            \gamma \\
            \vdots \\
            \gamma \\
            \gamma \\
            \delta
        \end{pmatrix}
    \end{equation*}
    for some $\alpha,\beta,\gamma,\delta\in\R$. The symmetry of~$M$ implies $m_{N,N-1}=m_{N-1,N}$ so that $\alpha=\gamma$. The constant column sum further gives $\beta=\delta$. As a result, $M$~acts the same on the spanning vectors of both~$S_{N-1}$ and~$S_N$. In particular, any fully synchronous equilibrium can either be fully stable or fully unstable in directions in both of these subspaces so that reciprocal heteroclinic connections are impossible.
\end{proof}

The first part in the proof of \Cref{thr:field-undirected} about realizing the Field cycle for edge-dependent coupling, however, does not necessarily hold for~$N>3$.
Specifically, homogeneity in every order~$m$ is forced for $N=3$ by restricting to synchrony subspaces in which all but one variable are equal; this forces the coupling functions to be equal.
For $N>3$ there are not sufficiently many conditions to deduce that all coupling functions are equal. Rather, the existence of invariant subspaces forces some coupling functions to be equal when all but one coupling argument are equal which does not necessarily imply for them to be equal as functions for arbitrary coupling arguments.
%but not necessarily all of them.
Thus, we cannot exclude the realization of the Field cycle for dynamics on any undirected hypergraph with edge-dependent coupling functions.

A similar argumentation applies to the more general Field construction with $N-1$~connections laid out in \Cref{subsubsec:fields-cycle} on undirected hypergraphs with $N>3$ vertices. 
\Cref{thr:field-undirected-large} implies that already two heteroclinic connections in two different synchrony subspaces in which one node desynchronizes are impossible when the coupling is undirected and homogeneous in each order. 
Thus, in particular larger heteroclinic structures following this construction are impossible. 
On the other hand, by the reasoning above we do not expect the presence of $N-1$ synchrony subspaces in which all but one cell are synchronized to force homogeneity in each order for $N>3$.
Hence, being able to realize more general heteroclinic structures through Field's construction cannot be excluded.

%%%%%%%
\section{For Which Classes of Network Dynamics on Hypergraphs Can We Realize Heteroclinic Cycles?}
\label{sec:HOI}
\noindent
The framework for network dynamical systems on hypergraphs described in \Cref{Sec:Setup} is sufficiently general to encompass a number of examples discussed in the literature. In the following we will consider explicit examples from the literature and assess---based on our results---whether the class of network dynamical systems on hypergraphs allows to realize the Guckenheimer--Holmes or Field heteroclinic cycle. 
{With the results in the previous section in mind, we restrict our attention again to hypergraphs with $N=3$ vertices.}

Note that these classes of network dynamical systems also restrict the set of coupling functions~$G_e$ associated to each hyperedge~$e$.
Thus, we assess here whether there are obstructions to the realization of either heteroclinic cycles rather than making a statement that there are parameters for which a heteroclinic cycle exists.

%%%
\subsection{Dynamics on Undirected Hypergraphs With Different Edge Types}

\newcommand{\ExM}{
\text{\cite{Mulas2020}}
}

In~\cite{Mulas2020}, Mulas and others consider dynamics on an undirected hypergraph~$\Hc^\ExM = (\Vc^\ExM, \Ec^\ExM)$, where the state~$x_k$ of node~$k$ evolves according to
\begin{equation}
    \label{eq:mulas_lit}
    \dot x_k = F(x_k) + \sum_{e\in\Ec: k\in e}G_{k;e}(x_e),
\end{equation}
where~$F$ determines the internal dynamics and~$G_{k;e}$ the target specific coupling function associated to each hyperedge. 
Note that while $\Hc^\ExM$ may be undirected, having a set of coupling functions assigned to each hyperedge corresponds to different edge `types' that effectively introduce directionality (e.g., some of them could vanish).

This general setup in our framework corresponds to a network dynamical system on a directed hypergraph~$\Hc = (\Vc, \Ec)$: 
Each edge $e = \{v_{j_1}, \dotsc, v_{j_m}\}\in\Ec^\ExM$ of order~$m+1$ generates~$m$ edges of~$\Ec$ of the form $(\{v_{j_1}, \dotsc, v_{j_m}\}, \{v_{j_k}\})$; the coupling is nodeunspecific.
While this setup allows to generate the Field cycle (cf.~\Cref{sec:oscar-directed}), the realization of the Guckenheimer--Holmes cycle in such network dynamical systems is not possible (\Cref{thm:NodeunspecificNoGH}).

The main results on synchronization of~\cite{Mulas2020} are stated for a class of coupling functions which are homogenous in both target node and edge order by assuming $G^{(m)}_{k;e}=G^{(m)}$ where~$m$ is the order of~$e$. As the possible directionality in the more general setup is lost, neither the Guckenheimer--Holmes nor the Field cycle can be realized (\Cref{Sec:Undirected}).

\subsection{Dynamics on Undirected Hypergraphs With Directed Coupling}

\newcommand{\ExS}{
\text{\cite{Salova2021a}}
}

Salova and D'Souza, in~\cite{Salova2021a}, consider dynamics on undirected hypergraphs $\Hc^\ExS = (\Vc^\ExS, \Ec^\ExS)$.
For a hyperedge~$e$ let~$|e|$ denote its order.
Now the dynamics of a node~$k\in\Vc^{\ExS}$ depend on the state of all nodes that are incident to it according to
\begin{equation}
    \label{eq:salova_lit}
    \dot x_k = F(x_k) + \sum_{\substack{e\in\Ec^\ExS:k\in e}}
    G^{(|e|)}(x_k; x_{e\sm k}),
\end{equation}
with a coupling function~$G^{(m)}$ that only depends on the order of the hyperedge~$e$---$e\sm k$ is shorthand notation for $e\sm\{k\}$.

Due to how~$G^{(m)}$ depends on the arguments, this dynamics corresponds in our setup to a network dynamical system on a directed hypergraph~$\Hc = (\Vc, \Ec)$ despite~$\Hc^{\ExS}$ being undirected: 
Each (undirected) hyperedge $e = \{v_{j_1}, \dotsc, v_{j_m}\}\in\Ec^\ExM$ generates~$m$ directed hyperedges of~$\Ec$ of order~$m$ that are of the form $(\{v_{j_1}, \dotsc, v_{j_m}\}\sm \{v_{j_k}\}, \{v_{j_k}\})$.
Moreover, the coupling is nodespecific but homogeneous in every hyperedge order.
Thus, network dynamical systems~\eqref{eq:salova_lit} can realize the Guckenheimer--Holmes as well as the Field cycle.

%%%
\subsection{Dynamics on Directed Weighted Hypergraphs}

\newcommand{\ExA}{
\text{\cite{Aguiar2020}}
}

Aguiar and coworkers~\cite{Aguiar2020} developed a coupled cell network framework for network dynamics with higher order interactions. 
They consider dynamics on directed hypergraphs~$\Hc^\ExA = (\Vc^\ExA, \Hc^\ExA)$ where each edge~$e\in\Hc^\ExA$ carries weight $w_e\in\R$; while they allow the tail~$T(e)$ to be a multiset, we restrict to the subclass where~$T(e)$ is a set.
The dynamics of node~$k$ is determined by
\begin{equation}
    \label{eq:porto_lit}
    \dot x_k = F(x_k) +
    \sum_{
    \substack{e\in\Hc^\ExA:k\in H(e)}
    }    
    w_{e} G^{(|e|)}\!\left(x_k; x_{T(e)} \right)
\end{equation}
This setup is quite general as the coupling can be nodespecific and allows for the realization of both the Guckenheimer--Holmes cycle as well as the Field cycle.

%%%
\subsection{Generalized Laplacian Dynamics on Undirected Hypergraphs}
\newcommand{\ExC}{
	\text{\cite{Carletti2020}}
}

In~\cite{Carletti2020}, Carletti and coauthors consider network dynamics on an undirected hypergraph $\Hc^\ExC = (\Vc^\ExC, \Ec^\ExC)$. 
The evolution of node~$k$ is determined by
\begin{equation}
	\label{eq:carletti_lit}
	\dot x_k = F(x_k) -\epsilon \sum_{e\in\Ec^\ExC: k\in e} \sum_{j\in e \sm k}(|e|-1)\left(G(x_k)-G(x_j)\right)
\end{equation}
for each node $k$, where $\epsilon$ denotes the coupling strength---as before $e\sm k$ is shorthand notation for $e\sm\{k\}$. Due to the specific form of the coupling functions, \eqref{eq:carletti_lit} describes generalized Laplacian dynamics. In fact, it can readily be seen that the internal sum can be considered to be taken over the index set $e$ instead of $e\sm k$. Furthermore, the coefficient depends only on the hyperedge order $m=|e|+1$. Thus, in our framework this dynamics corresponds to an undirected hypergraph $\Hc=(\Vc,\Ec)=\Hc^\ExC$ with nodespecific coupling that is homogeneous in every order given by the coupling functions
\[ \go{m}(x_k;x_e)=-\epsilon(m-2)\sum_{j\in e}\left(G(x_k)-G(x_j)\right) \]
(the number of terms in the sum is independent of the edge~
$e$). In particular, neither the Guckenheimer--Holmes cycle nor the Field cycle can be realized in systems of the form~\eqref{eq:carletti_lit}; cf.~\Cref{Sec:Undirected}.

%%%
\subsection{Replicator Equations on Directed Hypergraphs}
A type of network dynamics with higher-order interactions discussed in~\cite{Bairey2016} are replicator equations. 
These are related to generalized Lotka--Volterra equations describing interacting populations which support heteroclinic dynamics~\cite{Afraimovich2004c}. 
Specifically, the evolution of node~$k$ is given by
\begin{equation}
    \label{eq:rep_lit}
    \dot x_k = x_k\left(f_k(\x) - \sum_{j=1}^N x_jf_j(\x)\right)
\end{equation}
where $\x=(x_1,\dotsc,x_N)$ and
\[
f_l(\x) = -x_l + \sum_{r=1}^N a_{lr}x_r + \sum_{r=1}^N \sum_{p=1}^N b_{lrp}x_rx_p + \dotsb.
\]
Note that~\eqref{eq:rep_lit} contains terms that are quadratic or of higher polynomial degree only. As an immediate consequence, the Guckenheimer--Holmes system~\eqref{eq:gh_cubic} cannot be realized by systems of this form. 
Nonetheless, in its most general form, the equations model dynamics on a directed hypergraph with very general nodespecific couplings. 
Therefore, the possibility to realise the Field cycle with systems of the form~\eqref{eq:rep_lit} is to be expected despite the specific form of the equations of motion.

In a related context, consider the variations described in~\cite{Gibbs.2022} where node~$k$ evolves according to
\begin{align}
    \dot x_k &= x_k\left(R_k - \sum_{j=1}^NA_{kj}x_j -\sum_{j,l=1}^NB_{kjl}x_k \right)  \label{eq:gibbs_lit}\\
    \intertext{and, from~\cite{Grilli.2017},}
    \dot x_k &=x_k\sum_{j,l=1}^NP_{kjl}x_jx_l   \label{eq:grilli_lit}
\end{align}
 for ecosystems of multiple species exhibiting higher order interactions. Compared to~\eqref{eq:rep_lit} they do not contain the average fitness term (the second term in the parentheses). 
 Hence, there are no internal relations between the coupling coefficients. 
 It can readily be seen that~\eqref{eq:gibbs_lit} can realize the Guckenheimer--Holmes system~\eqref{eq:gh_cubic} while \eqref{eq:grilli_lit} cannot, since there are no linear terms in the equations of motion. 
 The coupling topology of both systems is that of a directed hypergraph with nodespecific coupling.
 The coupling functions are restricted to cubic polynomials. 

 Nonetheless, the Field cycle can be realized. 
 In fact it was shown in~\cite{Aguiar.2011} that the Field cycle exists in cubic systems. 
 For~\eqref{eq:grilli_lit} additional investigations have to be made to give a definite answer as to whether its realization is possible in a system with cubic terms only. 
 Finally, note that~\cite{Levine.2017} describes periodic fluctuations of species abundances in a model similar to~\eqref{eq:gibbs_lit} and~\eqref{eq:grilli_lit}. 
 This dynamical phenomenon can be organized by heteroclinic cycles.

%%%%%%%
\section{Discussion}
\label{sec:Discussion}
\noindent
The class of hypergraphs considered restricts what heteroclinic cycles can be realized in network dynamical systems with higher-order interactions. Specifically, we see that the vector fields one obtains through an underlying class of undirected hypergraphs (under the assumption that the coupling functions are compatible with the combinatorial structure) typically does not allow to realize the heteroclinic cycles considered here ({see~\Cref{Sec:Undirected,sec:LargerN}}). A key issue is that an undirected hypergraph imposes symmetries on the resulting vector field: Each hyperedge is a set of vertices which means that each node in the edge is affected by each other node in the same way (indecently of how they are ordered). The emergence of heteroclinic cycles is prevented by the symmetries, for example, by forcing the transverse linear stability in all directions to be the same or by allowing to get the required linearization of one equilibrium in the cycle but not all simultaneously.

Dynamics on directed (hyper)graphs provide more flexibility to realize heteroclinic cycles. 
{Such dynamical systems naturally arise as phase dynamics derived through phase reduction~\cite{Bick2023}.
While distinct frameworks for dynamics on directed hypergraphs have recently been developed~\cite{Aguiar2020,Gallo2022,vonderGracht.2023,vonderGracht.2023e}, the dynamical focus has predominantly been on the local dynamics around synchronized solutions.
At the same time, even though not stated explicitly, directed higher-order interactions do facilitate the emergence of heteroclinic structures in phase oscillator networks~\cite{Bick2017c,Bick2018}.
}%
Thus, while dynamics on directed hypergraphs have only received limited attention, 
we expect this to be the natural class of network dynamical systems that capture global dynamical phenomena in real world processes.

Whether a given heteroclinic cycle can be realized does not only depend on the underlying class of hypergraphs but also on the choice of interaction functions. As discussed in \Cref{sec:HOI}, certain model equations come with a natural choice of interaction function, such as ``Laplacian-like'' coupling in~\cite{Salova2021a}. Whether it is possible to realize a given heteroclinic cycle in the class of coupling functions is possible depends on the particular class; an analysis of this for each model in \Cref{sec:HOI} is beyond the scope of this article.

{Here we discussed realizing two key examples of heteroclinic cycles in network dynamical systems on hypergraphs.}
In particular, the Field cycle is an example of a more general construction to realize heteroclinic structures in dynamical systems. The analysis was restricted predominantly to the \emph{existence} of the heteroclinic structures in phase space (apart from some comments on the Guckenheimer--Holmes cycle).
Additional features of heteroclinic structures, such as (dynamical) stability, the number of connecting heteroclinic trajectories, or whether the stable manifolds of the equilibria contain all their unstable manifolds will necessarily lead to further conditions on the network structure and coupling functions. 
{While our focus was primarily on dynamics on hypergraphs on three vertices, we also briefly discussed our results in the context of hypergraphs with a larger number of nodes in \Cref{sec:LargerN}.
Understanding heteroclinic phenomena---including the general construction that the Field cycle is an example of---for network dynamics on hypergraphs in more detail yields opportunities for further research.
}

\section*{Acknowledgements}
\noindent
CB acknowledges support from the Engineering and Physical Sciences Research Council (EPSRC) through the grant EP/T013613/1.
SvdG was partially funded by the Deutsche Forschungsgemeinschaft (DFG, German Research Foundation)---453112019.

\printbibliography

\appendix
\section{The Guckenheimer--Holmes Cycle in Directed \texorpdfstring{$3$}{3}-Uniform Hypergraphs}
\label{sec:app-gh-3-uniform}
\newcommand{\e}{\Pi}
\newcommand{\ee}{\Phi}
\newcommand{\eee}{\Psi}
\newcommand{\Rnum}[1]{\MakeUppercase{\romannumeral #1}}
\noindent
In this appendix, we fill the gaps in the considerations on realizing the Gucken\-heimer--Holmes system \eqref{eq:gh_cubic} as a network dynamical system on directed $3$-uniform hypergraphs by completing the proof of \Cref{thr:gh-3-uniform}. 
\begin{proof}[Proof of \Cref{thr:gh-3-uniform}, continued]
    The expression in \eqref{eq:gh-directed-uniform} has to be equal to the governing function $f(z,y_1,y_2)=z+az^3+bzy_1^2+czy_2^2$ from \eqref{eq:gh_cubic}. That is, we have
    \begin{equation}
        \label{eq:app-gh-param}
        \begin{split}
            1 &= \alpha_1 \\
            a &= \alpha_2 + \beta (\ee+\eee) \\
            b &= \beta (\e+\ee) \\
            c &= \beta (\e+\eee)
        \end{split}
    \end{equation}
    and $P\equiv 0, P' \equiv 0$. In fact, it suffices to choose $P, P'$ such that $Q\equiv0$, however, the more restrictive assumption does not change the argumentation. In the remainder, we abbreviate $\alpha=\alpha_2$. For the Guckenheimer--Holmes cycle to emerge in the system, the parameters need to satisfy the following inequalities.
    \begin{align}
        & a+b+c = -1 \label{eq:app-gh-1} \\
        & -\frac{1}{3} < a < 0 \label{eq:app-gh-2} \\
        & c < a < b < 0 \label{eq:app-gh-3}.
    \end{align}
    As a first observation, \eqref{eq:app-gh-3} implies that $\e, \ee, \eee$ cannot vanish simultaneously. This fact is frequently used in the upcoming transformations without mention. Substitute $a,b,c$ from \eqref{eq:app-gh-param} into \eqref{eq:app-gh-1} to obtain
    \[ \beta = - \frac{\alpha+1}{2(\e+\ee+\eee)}. \]
    Substituting this back into \eqref{eq:app-gh-param}, we obtain
    \begin{equation*}
        \begin{split}
            a &= \alpha - \frac{(\alpha+1)(\ee+\eee)}{2(\e+\ee+\eee)}  \\
            b &= - \frac{(\alpha+1)(\ee+\eee)}{2(\e+\e+\ee)} \\
            c &= - \frac{(\alpha+1)(\ee+\eee)}{2(\e+\e+\eee)}
        \end{split}
    \end{equation*}
    
    We now substitute these expressions into in the inequalities to obtain four inequalities for $\alpha$ depending on $\e, \ee, \eee$. First, substitute $a$ into \eqref{eq:app-gh-2}. Omitting any details, this can equivalently be transformed into
    \begin{equation*}
        \frac{-2\e+\ee+\eee}{6\e+3\ee+3\eee} < \alpha < \frac{3\ee+3\eee}{6\e+3\ee+3\eee}.
    \end{equation*}
    We refer to these two inequalities as \Rnum{1} and \Rnum{2}. Substituting $a,b,c$ into \eqref{eq:app-gh-3}, the first two inequalities can equivalently be transformed into
    \begin{equation*}
        -\frac{\e-\ee}{3\e+\ee+2\eee} < \alpha < -\frac{\e-\eee}{3\e+2\ee+\eee},
    \end{equation*}
    which we refer to as inequalities \Rnum{3} and \Rnum{4}. The third inequality in \eqref{eq:app-gh-3} is equivalent to $\alpha>-1$ if $\e+\ee\ne0$, which we see to be true below. This inequality is automatically satisfied whenever \Rnum{1} is satisfied. These considerations show that the cubic Guckenheimer--Holmes system \eqref{eq:gh_cubic} can be realized if and only if the hypergraph is such that $\e,\ee,\eee$ are independent of the targeted vertex and the inequalities \Rnum{1}--\Rnum{4} can be satisfied simultaneously.
    
    In the next step, we investigate, which configurations of $\e, \ee, \eee$ are possible, if these conditions are satisfied. The first immediate observations are that the third inequality in \eqref{eq:app-gh-3}---$b<0$---can only be satisfied if
    \begin{equation}
        \label{eq:app-gh-cond1}
        \e+\ee\ne0.
    \end{equation}
    and that \Rnum{3} and \Rnum{4} can only both be satisfied if
    \begin{equation}
        \label{eq:app-gh-cond2}
        \ee<\eee.
    \end{equation}
    All hyperedges in a $3$-uniform hypergraph on three vertices stem from the set
    \begin{equation*}
        \left\{
            \begin{array}{lll}
                \he{1,2}{1}, & \he{1,2}{1,2}, & \he{1,2}{1,2,3} \\
                \he{1,3}{1}, & \he{1,3}{1,2}, & \he{1,3}{1,2,3} \\
                \he{2,3}{1}, & \he{2,3}{1,2}, & \he{2,3}{1,2,3} \\
                \he{1,2}{2}, & \he{1,2}{1,3} & \\
                \he{1,3}{2}, & \he{1,3}{1,3} & \\
                \he{2,3}{2}, & \he{2,3}{1,3} & \\
                \he{1,2}{3}, & \he{1,2}{2,3} & \\
                \he{1,3}{3}, & \he{1,3}{2,3} & \\
                \he{2,3}{3}, & \he{2,3}{2,3} &
            \end{array}
        \right\}.
    \end{equation*}
    We define subsets $\omega^0_k$ for $k=1,2,3$ of hyperedges that are a true $2$-to-$1$ input for vertex $k$:
    \begin{align*}
        \omega^0_1 & = \{ \he{2,3}{1}, \he{2,3}{1,2}, \he{2,3}{1,3}, \he{2,3}{1,2,3} \} \\
        \omega^0_2 &= \{ \he{1,3}{2}, \he{1,3}{1,2}, \he{1,3}{2,3}, \he{1,3}{1,2,3} \} \\
        \omega^0_3 &= \{ \he{1,2}{3}, \he{1,2}{1,3}, \he{1,2}{2,3}, \he{1,2}{1,2,3} \}.
    \end{align*}
    That is, $E_0\subset \omega^0_k$ for vertex $k$. Similarly, we define
    \begin{align*}
        \omega^1_1 & = \{ \he{1,2}{1}, \he{1,2}{1,2}, \he{1,2}{1,3}, \he{1,2}{1,2,3} \} \\
        \omega^1_2 &= \{ \he{2,3}{2}, \he{2,3}{1,2}, \he{2,3}{2,3}, \he{2,3}{1,2,3} \} \\
        \omega^1_3 &= \{ \he{1,3}{3}, \he{1,3}{1,3}, \he{1,3}{2,3}, \he{1,3}{1,2,3} \}
    \end{align*}
    to be the degenerate inputs from the right neighbor, i.e., $E_1 \subset \omega^1_k$ for vertex $k$, and
    \begin{align*}
        \omega^2_1 & = \{ \he{1,3}{1}, \he{1,3}{1,2}, \he{1,3}{1,3}, \he{1,3}{1,2,3} \} \\
        \omega^2_2 &= \{ \he{1,2}{2}, \he{1,2}{1,2}, \he{1,2}{2,3}, \he{1,2}{1,2,3} \} \\
        \omega^2_3 &= \{ \he{2,3}{3}, \he{2,3}{1,3}, \he{2,3}{2,3}, \he{2,3}{1,2,3} \}
    \end{align*}
    to be the degenerate inputs from the left neighbor, i.e., $E_2 \subset \omega^2_k$ for vertex $k$. We immediately observe
    \begin{equation}
        \label{eq:app-gh-cond3}
        |E_i| \le 4 \quad \text{for}\quad i=0,1,2.
    \end{equation}
    As these subsets intersect non-trivially, we cannot arbitrarily distribute hyperedges to generate arbitrary combinations of $\e,\ee,\eee$. In fact, we observe that the true $2$-to-$1$ connections are all in exactly one of those sets, the $2$-to-$2$ connections are in $\omega_k^i\cap\omega_l^j$ for some $k\ne l$ and $i\ne j$, and the $2$-to-$3$ connections are in $\omega_1^i\cap\omega_2^j\cap\omega_3^s$ for $\{i,j,s\}=\{0,1,2\}$.
    
    Since $\e,\ee,\eee$ are necessarily independent of the vertex $k$, it is handy to investigate the sets $\Omega_i = \omega_1^i\cup\omega_2^i\cup\omega_3^i$ for $i=0,1,2$ to deduce the restrictions. In fact, these sets have the following relations
    \begin{align*}
        \Omega_0\cap\Omega_1 &= 
        \left\{
            \begin{array}{ll}
                \he{1,2}{1,3}, & \he{1,2}{1,2,3}, \\
                \he{2,3}{1,2}, & \he{2,3}{1,2,3}, \\
                \he{1,3}{2,3}, & \he{1,3}{1,2,3}
            \end{array}
        \right\}, \\[5pt]
        \Omega_0\setminus\Omega_1 &= 
        \left\{
            \begin{array}{ll}
                \he{2,3}{1}, & \he{2,3}{1,3}, \\
                \he{1,3}{2}, & \he{1,3}{1,2}, \\
                \he{1,2}{3}, & \he{1,2}{2,3}
            \end{array}
        \right\}, \\[5pt]
        \Omega_1\setminus\Omega_0 &= 
        \left\{
            \begin{array}{ll}
                \he{1,2}{1}, & \he{1,2}{1,2}, \\
                \he{2,3}{2}, & \he{2,3}{2,3}, \\
                \he{1,3}{3}, & \he{1,3}{1,3}
            \end{array}
        \right\}, \\[5pt]
        \Omega_0\cap\Omega_2 &= 
        \left\{
            \begin{array}{ll}
                \he{1,3}{1,2}, & \he{1,3}{1,2,3}, \\
                \he{1,2}{2,3}, & \he{1,2}{1,2,3}, \\
                \he{2,3}{1,3}, & \he{2,3}{1,2,3}
            \end{array}
        \right\}, \\[5pt]
        \Omega_0\setminus\Omega_2 &= 
        \left\{
            \begin{array}{ll}
                \he{2,3}{1}, & \he{2,3}{1,2}, \\
                \he{1,3}{2}, & \he{1,3}{2,3}, \\
                \he{1,2}{3}, & \he{1,2}{1,3}
            \end{array}
        \right\}, \\[5pt]
        \Omega_2\setminus\Omega_0 &= 
        \left\{
            \begin{array}{ll}
                \he{1,3}{1}, & \he{1,3}{1,3}, \\
                \he{1,2}{2}, & \he{1,2}{1,2}, \\
                \he{2,3}{3}, & \he{2,3}{2,3}
            \end{array}
        \right\}, \\[5pt]
        \Omega_1\cap\Omega_2 &= 
        \left\{
            \begin{array}{ll}
                \he{1,3}{1,3}, & \he{1,3}{1,2,3}, \\
                \he{1,2}{1,2}, & \he{1,2}{1,2,3}, \\
                \he{2,3}{2,3}, & \he{2,3}{1,2,3}
            \end{array}
        \right\}, \\[5pt]
        \Omega_1\setminus\Omega_2 &= 
        \left\{
            \begin{array}{ll}
                \he{1,2}{1}, & \he{1,2}{1,3}, \\
                \he{2,3}{2}, & \he{2,3}{1,2}, \\
                \he{1,3}{3}, & \he{1,3}{2,3}
            \end{array}
        \right\}, \\[5pt]
        \Omega_2\setminus\Omega_1 &= 
        \left\{
            \begin{array}{ll}
                \he{1,3}{1}, & \he{1,3}{1,2}, \\
                \he{1,2}{2}, & \he{1,2}{2,3}, \\
                \he{2,3}{3}, & \he{2,3}{1,3}
            \end{array}
        \right\}
    \end{align*}
    Using these sets, we can deduce restrictions on $\e,\ee,\eee$ by the following combinatorical considerations:
    \begin{itemize}
        \item If $|E_i|=2$ the corresponding set $\Omega_i$ needs to contain elements besides the $2$-to-$1$ connections. Thus, $\Omega_i\cap\Omega_j \ne\emptyset$ for some $j\ne i$. This implies $|E_j|\ge1$. To summarize
        \begin{equation}
            \label{eq:app-gh-cond4}
            |E_i|=2 \implies |E_j|\ge1 \text{ for some } j\ne i.
        \end{equation}
        \item A similar argument shows that whenever $|E_i|\ge3$ there are elements in both intersections $\Omega_i\cap\Omega_j$ and $\Omega_i\cap\Omega_s$ for $\{i,j,s\}=\{0,1,2\}$, which are therefore non-empty. This can be summarized as
        \begin{equation*}
            |E_i|\ge3 \implies |E_j|\ge1 \text{ for all } j\ne i.
        \end{equation*}
        \item Similarly, whenever $|E_i|=4$ all elements in both intersections $\Omega_i\cap\Omega_j$ and $\Omega_i\cap\Omega_k$ for $\{i,j,s\}=\{0,1,2\}$ are present. Hence,
        \begin{equation*}
            |E_i|=4 \implies |E_j|\ge2 \text{ for all } j\ne i.
        \end{equation*}
        \item The previous two restrictions can further be summarized as
        \begin{equation}
            \label{eq:app-gh-cond5}
            \left| |E_i| - |E_j| \right| \le 2 \text{ for all } i,j.
        \end{equation}
        \item Whenever $\e=\eee=4$ all $2$-to-$2$ and all $2$-to-$3$ connections must be present. These are all elements in $\Omega_0\cap\Omega_1$ and $\Omega_1\cap\Omega_2$. In particular, we have
        \begin{equation}
            \label{eq:app-gh-cond6}
            \e=\eee=4 \implies \ee\ge3.
        \end{equation}
    \end{itemize}
    
    There are exactly $20$ combinations of $\e,\ee,\eee$ that satisfy the conditions in \Cref{eq:app-gh-cond1,eq:app-gh-cond2,eq:app-gh-cond3,eq:app-gh-cond4,eq:app-gh-cond5,eq:app-gh-cond6}. In fact, all of these configurations can be realized by a $3$-uniform hypergraph. Furthermore, inequalities \Rnum{1}--\Rnum{4} can be satisfied for each of these configurations. Thus, the Guckenheimer--Holmes system \eqref{eq:gh_cubic} can be realized. We list all possible configurations and inequalities together with an example of a suitable hypergraph in \Cref{tab:ap-gh-uniform} below. In particular, the cubic Guckenheimer--Holmes system with admissible parameters is realized for any of those hypergraphs if $\alpha$ is chosen to satisfy the inequalities in the second column and $\beta= - \frac{\alpha+1}{2(\e+\ee+\eee)}$ and as a result $f$ and $\go{3}$ are as presented above. Note that there can be examples of a hypergraph realizing a particular configuration of $\e, \ee, \eee$ besides the one that is listed in the table. 
\end{proof}

\renewcommand*{\arraystretch}{1.4}
\begin{longtable}[c]{c|c|l}
    \caption{List of all configurations of $\e, \ee, \eee$ that allow for the realization of the cubic Guckenheimer--Holmes system \eqref{eq:gh_cubic} as a network dynamical system on a $3$-uniform hypergraph. The second column specifies the range in which the parameter $\alpha$ satisfies inequalities \Rnum{1}--\Rnum{4}. The third column lists one example of a $3$-uniform hypergraph that realizes the given configuration of $\e,\ee,\eee$.\label{tab:ap-gh-uniform}}\\
    \hline
    \multicolumn{3}{| c |}{Begin of Table}\\
    \hline
    Parameters & Inequalities \Rnum{1}--\Rnum{4} & Hyperedges \\
    \hline
    \endfirsthead

    \hline
    \multicolumn{3}{|c|}{Continuation of \Cref{tab:ap-gh-uniform}}\\
    \hline
    Parameters & Inequalities & Hyperedges \\
    \hline
    \endhead

    \hline
    \endfoot

    \hline
    \multicolumn{3}{| c |}{End of Table}\\
    \hline\hline
    \endlastfoot

            $\begin{array}{l} \e=0 \\ \ee=1 \\ \eee=2\end{array}$
        &
            $\begin{array}{c} \frac{1}{5}<\alpha<\frac{1}{2} \\[2pt] \frac{1}{3}<\alpha<1 \end{array}$
        &
            $\begin{array}{ll}
                \he{1,3}{1}, & \he{1,3}{1,3}, \\
                \he{1,2}{2}, & \he{1,2}{1,2}, \\
                \he{2,3}{3}, & \he{2,3}{2,3}
            \end{array}$
        \\[5pt] \hline

            $\begin{array}{l} \e=1 \\ \ee=0 \\ \eee=1\end{array}$
        &
            $\begin{array}{c} -\frac{1}{5}<\alpha<0 \\[2pt] -\frac{1}{9}<\alpha<\frac{1}{3} \end{array}$
        &
            $\begin{array}{ll}
                \he{2,3}{1}, & \he{1,3}{1}, \\
                \he{1,3}{2}, & \he{1,2}{1}, \\
                \he{1,2}{3}, & \he{2,3}{2}
            \end{array}$
        \\[5pt] \hline

            $\begin{array}{l} \e=1 \\ \ee=0 \\ \eee=2\end{array}$
        &
            $\begin{array}{c} -\frac{1}{7}<\alpha<\frac{1}{5} \\[2pt] 0<\alpha<\frac{1}{2} \end{array}$
        &
            $\begin{array}{ll}
                \he{1,3}{1}, & \he{1,3}{1,2}, \\
                \he{1,2}{2}, & \he{1,2}{2,3}, \\
                \he{2,3}{3}, & \he{2,3}{1,3}
            \end{array}$
        \\[5pt] \hline

            $\begin{array}{l} \e=1 \\ \ee=1 \\ \eee=2\end{array}$
        &
            $\begin{array}{c} 0<\alpha<\frac{1}{7} \\[2pt] \frac{1}{15}<\alpha<\frac{3}{5} \end{array}$
        &
            $\begin{array}{ll}
                \he{1,3}{1}, & \he{1,3}{1,2}, \\
                \he{1,2}{2}, & \he{1,2}{2,3}, \\
                \he{2,3}{3}, & \he{2,3}{1,3}, \\
                \he{1,2}{1}, & \\
                \he{2,3}{2}, & \\
                \he{1,3}{3}  & 
            \end{array}$
        \\[5pt] \hline

            $\begin{array}{l} \e=1 \\ \ee=1 \\ \eee=3\end{array}$
        &
            $\begin{array}{c} 0<\alpha<\frac{1}{4} \\[2pt] \frac{1}{9}<\alpha<\frac{2}{3} \end{array}$
        &
            $\begin{array}{ll}
                \he{1,3}{1}, & \he{1,3}{1,2}, \\
                \he{1,2}{2}, & \he{1,2}{2,3}, \\
                \he{2,3}{3}, & \he{2,3}{1,3}, \\
                 & \he{1,3}{1,3}, \\
                 & \he{1,2}{1,2}, \\
                 & \he{2,3}{2,3}
            \end{array}$
        \\[5pt] \hline

            $\begin{array}{l} \e=1 \\ \ee=2 \\ \eee=3\end{array}$
        &
            $\begin{array}{c} \frac{1}{11}<\alpha<\frac{1}{5} \\[2pt] \frac{1}{7}<\alpha<\frac{5}{7} \end{array}$
        &
            $\begin{array}{ll}
                \he{1,3}{1}, & \he{1,3}{1,2}, \\
                \he{1,2}{2}, & \he{1,2}{2,3}, \\
                \he{2,3}{3}, & \he{2,3}{1,3}, \\
                \he{1,2}{1}, & \he{1,3}{1,3}, \\
                \he{2,3}{2}, & \he{1,2}{1,2}, \\
                \he{1,3}{3}, & \he{2,3}{2,3}
            \end{array}$
        \\[5pt] \hline

            $\begin{array}{l} \e=2 \\ \ee=0 \\ \eee=1\end{array}$
        &
            $\begin{array}{c} -\frac{1}{4}<\alpha<-\frac{1}{7} \\[2pt] -\frac{1}{5}<\alpha<\frac{1}{5} \end{array}$
        &
            $\begin{array}{ll}
                \he{2,3}{1}, & \he{1,3}{1,2}, \\
                \he{1,3}{2}, & \he{1,2}{2,3}, \\
                \he{1,2}{3}, & \he{2,3}{1,3}
            \end{array}$
        \\[5pt] \hline

            $\begin{array}{l} \e=2 \\ \ee=0 \\ \eee=2\end{array}$
        &
            $\begin{array}{c} -\frac{1}{5}<\alpha<0 \\[2pt] -\frac{1}{9}<\alpha<\frac{1}{3} \end{array}$
        &
            $\begin{array}{ll}
                \he{2,3}{1}, & \he{1,3}{1,2}, \\
                \he{1,3}{2}, & \he{1,2}{2,3}, \\
                \he{1,2}{3}, & \he{2,3}{1,3}, \\
                \he{1,3}{1}, & \\
                \he{1,2}{2}, & \\
                \he{2,3}{3}
            \end{array}$
        \\[5pt] \hline

            $\begin{array}{l} \e=2 \\ \ee=1 \\ \eee=2\end{array}$
        &
            $\begin{array}{c} -\frac{1}{11}<\alpha<0 \\[2pt] -\frac{1}{21}<\alpha<\frac{3}{7} \end{array}$
        &
            $\begin{array}{ll}
                \he{2,3}{1}, & \he{1,3}{1,2}, \\
                \he{1,3}{2}, & \he{1,2}{2,3}, \\
                \he{1,2}{3}, & \he{2,3}{1,3}, \\
                \he{1,3}{1}, & \\
                \he{1,2}{2}, & \\
                \he{2,3}{3}, & \\
                \he{1,2}{1}, & \\
                \he{2,3}{2}, & \\
                \he{1,3}{3}
            \end{array}$
        \\[5pt] \hline

            $\begin{array}{l} \e=2 \\ \ee=1 \\ \eee=3\end{array}$
        &
            $\begin{array}{c} -\frac{1}{13}<\alpha<\frac{1}{11} \\[2pt] 0<\alpha<\frac{1}{2} \end{array}$
        &
            $\begin{array}{ll}
                \he{2,3}{1}, & \he{1,3}{1,2}, \\
                \he{1,3}{2}, & \he{1,2}{2,3}, \\
                \he{1,2}{3}, & \he{2,3}{1,3}, \\
                \he{1,3}{1}, & \he{1,3}{1,3}, \\
                \he{1,2}{2}, & \he{1,2}{1,2}, \\
                \he{2,3}{3}, & \he{2,3}{2,3}
            \end{array}$
        \\[5pt] \hline

            $\begin{array}{l} \e=2 \\ \ee=2 \\ \eee=3\end{array}$
        &
            $\begin{array}{c} 0<\alpha<\frac{1}{13} \\[2pt] \frac{1}{27}<\alpha<\frac{5}{9} \end{array}$
        &
            $\begin{array}{ll}
                \he{2,3}{1}, & \he{1,3}{1,2}, \\
                \he{1,3}{2}, & \he{1,2}{2,3}, \\
                \he{1,2}{3}, & \he{2,3}{1,3}, \\
                \he{1,3}{1}, & \he{1,3}{1,3}, \\
                \he{1,2}{2}, & \he{1,2}{1,2}, \\
                \he{2,3}{3}, & \he{2,3}{2,3}, \\
                \he{1,2}{1}, & \\
                \he{2,3}{2}, & \\
                \he{1,3}{3}
            \end{array}$
        \\[5pt] \hline

            $\begin{array}{l} \e=2 \\ \ee=2 \\ \eee=4\end{array}$
        &
            $\begin{array}{c} 0<\alpha<\frac{1}{7} \\[2pt] \frac{1}{15}<\alpha<\frac{3}{5} \end{array}$
        &
            $\begin{array}{lll}
                \he{1,3}{1}, & \he{1,3}{1,3}, & \he{1,3}{1,2,3}, \\
                \he{1,2}{2}, & \he{1,2}{1,2}, & \he{1,2}{1,2,3}, \\
                \he{2,3}{3}, & \he{2,3}{2,3}, & \he{2,3}{1,2,3}, \\
                 & \he{1,3}{1,2}, & \\
                 & \he{1,2}{2,3}, & \\
                 & \he{2,3}{1,3}
            \end{array}$
        \\[5pt] \hline

            $\begin{array}{l} \e=2 \\ \ee=3 \\ \eee=4\end{array}$
        &
            $\begin{array}{c} \frac{1}{17}<\alpha<\frac{1}{8} \\[2pt] \frac{1}{11}<\alpha<\frac{7}{11} \end{array}$
        &
            $\begin{array}{lll}
                \he{1,3}{1}, & \he{1,3}{1,3}, & \he{1,3}{1,2,3}, \\
                \he{1,2}{2}, & \he{1,2}{1,2}, & \he{1,2}{1,2,3}, \\
                \he{2,3}{3}, & \he{2,3}{2,3}, & \he{2,3}{1,2,3}, \\
                \he{1,2}{1}, & \he{1,3}{1,2}, & \\
                \he{2,3}{2}, & \he{1,2}{2,3}, & \\
                \he{1,3}{3}, & \he{2,3}{1,3}
            \end{array}$
        \\[5pt] \hline

            $\begin{array}{l} \e=3 \\ \ee=1 \\ \eee=2\end{array}$
        &
            $\begin{array}{c} -\frac{1}{7}<\alpha<-\frac{1}{13} \\[2pt] -\frac{1}{9}<\alpha<\frac{1}{3} \end{array}$
        &
            $\begin{array}{ll}
                \he{2,3}{1}, & \he{1,3}{1,2}, \\
                \he{1,3}{2}, & \he{1,2}{2,3}, \\
                \he{1,2}{3}, & \he{2,3}{1,3}, \\
                \he{1,3}{1}, & \he{1,2}{1,3}, \\
                \he{1,2}{2}, & \he{2,3}{1,2}, \\
                \he{2,3}{3}, & \he{1,3}{2,3}
            \end{array}$
        \\[5pt] \hline

            $\begin{array}{l} \e=3 \\ \ee=1 \\ \eee=3\end{array}$
        &
            $\begin{array}{c} -\frac{1}{8}<\alpha<0 \\[2pt] -\frac{1}{15}<\alpha<\frac{2}{5} \end{array}$
        &
            $\begin{array}{lll}
                \he{2,3}{1}, & \he{1,3}{1,2}, & \he{1,3}{1,2,3}, \\
                \he{1,3}{2}, & \he{1,2}{2,3}, & \he{1,2}{1,2,3}, \\
                \he{1,2}{3}, & \he{2,3}{1,3}, & \he{2,3}{1,2,3}, \\
                \he{1,3}{1}, & & \\
                \he{1,2}{2}, & & \\
                \he{2,3}{3} & &
            \end{array}$
        \\[5pt] \hline

            $\begin{array}{l} \e=3 \\ \ee=2 \\ \eee=3\end{array}$
        &
            $\begin{array}{c} -\frac{1}{17}<\alpha<0 \\[2pt] -\frac{1}{33}<\alpha<\frac{5}{11} \end{array}$
        &
            $\begin{array}{lll}
                \he{2,3}{1}, & \he{1,3}{1,2}, & \he{1,3}{1,2,3}, \\
                \he{1,3}{2}, & \he{1,2}{2,3}, & \he{1,2}{1,2,3}, \\
                \he{1,2}{3}, & \he{2,3}{1,3}, & \he{2,3}{1,2,3}, \\
                \he{1,3}{1}, & & \\
                \he{1,2}{2}, & & \\
                \he{2,3}{3}, & & \\
                \he{1,2}{1}, & & \\
                \he{2,3}{2}, & & \\
                \he{1,3}{3} & &                
            \end{array}$
        \\[5pt] \hline

            $\begin{array}{l} \e=3 \\ \ee=2 \\ \eee=4\end{array}$
        &
            $\begin{array}{c} -\frac{1}{19}<\alpha<\frac{1}{17} \\[2pt] 0<\alpha<\frac{1}{2} \end{array}$
        &
            $\begin{array}{lll}
                \he{2,3}{1}, & \he{1,3}{1,2}, & \he{1,3}{1,2,3}, \\
                \he{1,3}{2}, & \he{1,2}{2,3}, & \he{1,2}{1,2,3}, \\
                \he{1,2}{3}, & \he{2,3}{1,3}, & \he{2,3}{1,2,3}, \\
                \he{1,3}{1}, & \he{1,3}{1,3}, & \\
                \he{1,2}{2}, & \he{1,2}{1,2}, & \\
                \he{2,3}{3}, & \he{2,3}{2,3} &               
            \end{array}$
        \\[5pt] \hline

            $\begin{array}{l} \e=3 \\ \ee=3 \\ \eee=4\end{array}$
        &
            $\begin{array}{c} 0<\alpha<\frac{1}{19} \\[2pt] \frac{1}{39}<\alpha<\frac{7}{13} \end{array}$
        &
            $\begin{array}{lll}
                \he{2,3}{1}, & \he{1,3}{1,2}, & \he{1,3}{1,2,3}, \\
                \he{1,3}{2}, & \he{1,2}{2,3}, & \he{1,2}{1,2,3}, \\
                \he{1,2}{3}, & \he{2,3}{1,3}, & \he{2,3}{1,2,3}, \\
                \he{1,3}{1}, & \he{1,3}{1,3}, & \\
                \he{1,2}{2}, & \he{1,2}{1,2}, & \\
                \he{2,3}{3}, & \he{2,3}{2,3}, & \\
                \he{1,2}{1}, & & \\
                \he{2,3}{2}, & & \\
                \he{1,3}{3}
            \end{array}$
        \\[5pt] \hline

            $\begin{array}{l} \e=4 \\ \ee=2 \\ \eee=3\end{array}$
        &
            $\begin{array}{c} -\frac{1}{10}<\alpha<-\frac{1}{19} \\[2pt] -\frac{1}{13}<\alpha<\frac{5}{13} \end{array}$
        &
            $\begin{array}{lll}
                \he{2,3}{1}, & \he{1,3}{1,2}, & \he{1,3}{1,2,3}, \\
                \he{1,3}{2}, & \he{1,2}{2,3}, & \he{1,2}{1,2,3}, \\
                \he{1,2}{3}, & \he{2,3}{1,3}, & \he{2,3}{1,2,3}, \\
                \he{1,3}{1}, & \he{1,3}{2,3}, & \\
                \he{1,2}{2}, & \he{1,2}{1,3}, & \\
                \he{2,3}{3}, & \he{2,3}{1,2} & \\
            \end{array}$
        \\[5pt] \hline

            $\begin{array}{l} \e=4 \\ \ee=3 \\ \eee=4\end{array}$
        &
            $\begin{array}{c} -\frac{1}{23}<\alpha<0 \\[2pt] -\frac{1}{45}<\alpha<\frac{7}{15} \end{array}$
        &
            $\begin{array}{lll}
                \he{2,3}{1}, & \he{1,3}{1,2}, & \he{1,3}{1,2,3}, \\
                \he{1,3}{2}, & \he{1,2}{2,3}, & \he{1,2}{1,2,3}, \\
                \he{1,2}{3}, & \he{2,3}{1,3}, & \he{2,3}{1,2,3}, \\
                \he{1,3}{1}, & \he{1,3}{2,3}, & \\
                \he{1,2}{2}, & \he{1,2}{1,3}, & \\
                \he{2,3}{3}, & \he{2,3}{1,2}, & \\
                 & \he{1,3}{1,3}, & \\
                 & \he{1,2}{1,2}, & \\
                 & \he{2,3}{2,3}
            \end{array}$
\end{longtable}

\end{document}